\documentclass[final,10pt]{amsart}
\usepackage{amsmath,amsfonts,amssymb,amscd,verbatim}
\usepackage[margin=0.4in]{geometry}
\usepackage[leqno]{amsmath}
\usepackage{fullpage}
\usepackage{color}
\usepackage{dsfont}
\usepackage{tikz-cd}
\usepackage{cancel}
\usepackage{mathrsfs}
\usepackage{mathtools}

\makeatletter
\@namedef{subjclassname@2020}{\textup{2020} Mathematics Subject Classification}
\makeatother

\usepackage{enumerate}

\usepackage{bm}

\makeatletter
\newcommand{\leqnomode}{\tagsleft@true}
\newcommand{\reqnomode}{\tagsleft@false}
\makeatother

\numberwithin{equation}{section}

\theoremstyle{definition}
\newtheorem{theorem}{Theorem} 
\newtheorem*{theorem*}{Theorem}
\numberwithin{theorem}{section}

\newtheorem{corollary}[theorem]{Corollary} 
\newtheorem{lemma}[theorem]{Lemma}
\newtheorem{proposition}[theorem]{Proposition}
\newtheorem{question}[theorem]{Question}
\newtheorem{hypothesis}[theorem]{Hypothesis} 
\newtheorem{remark}[theorem]{Remark}
\newtheorem{definition}[theorem]{Definition} 
\newtheorem{example}[theorem]{Example}
\newtheorem{notation}[theorem]{Notation} 

\DeclareMathOperator\Aut{Aut}
\DeclareMathOperator\Ann{Ann}

\DeclareMathOperator\End{End}
\DeclareMathOperator\GKdim{GKdim}

\DeclareMathOperator\id{id}

\DeclareMathOperator\Maxspec{Maxspec}
\DeclareMathOperator\orb{orb}
\DeclareMathOperator\ord{ord}

\DeclareMathOperator\Supp{Supp}
\DeclareMathOperator\Spec{Spec}

\DeclareMathOperator\trdeg{tr.deg}

\renewcommand{\Bbb}{\mathbb}
\newcommand\CC{\mathbb C}
\newcommand\NN{\mathbb N}
\newcommand\PP{\mathbb P}

\newcommand\ZZ{\mathbb Z}

\newcommand\cO{\mathcal O}

\newcommand\frm{\mathfrak{m}}
\newcommand\frn{\mathfrak{n}}

\newcommand\fsl{\mathfrak{sl}}

\newcommand\bp{\mathbf p}
\newcommand\bq{\mathbf q}

\newcommand\cnt{\mathcal Z}
\newcommand\css{\mathcal S}

\newcommand\inv{^{-1}}
\newcommand\iso{\cong}
\newcommand\kk{\mathds{k}}

\newcommand{\sto}{\ensuremath{\rightarrow}}
\renewcommand{\to}{\ensuremath{\longrightarrow}}

\mathchardef\mhyphen="2D
\newcommand{\Frac}{\operatorname{Frac}}

\newcommand{\wmod}{\mathrm{\mhyphen wmod}}
\newcommand{\wmodO}{\mathrm{\mhyphen wmod_{\cO}}}

\newcommand\grp[1]{{\langle #1 \rangle}}

\definecolor{darkgreen}{RGB}{55,138,0}
\definecolor{orange}{RGB}{255,120,0}

\title{Weight modules over Bell--Rogalski algebras}

\author[Gaddis]{Jason Gaddis}
\address{(Gaddis) Department of Mathematics,  Miami University, Oxford, Ohio 45056} 
\email{gaddisj@miamioh.edu}

\author[Rosso]{Daniele Rosso}
\address{(Rosso) Department of Mathematics and Actuarial Science, Indiana University Northwest, Gary, IN 46408} 
\email{drosso@iu.edu}

\author[Won]{Robert Won}
\address{(Won) Department of Mathematics, The George Washington University, Washington, DC 20052}
\email{robertwon@gwu.edu}


\begin{document}

\begin{abstract}
We study a class of $\ZZ$-graded algebras introduced by Bell and Rogalski. Their construction generalizes in large part that of rank one generalized Weyl algebras (GWAs). We establish certain ring-theoretic properties of these algebras and study their connection to GWAs. We classify the simple weight modules in the infinite orbit case and provide a partial classification in the case of orbits of finite order.
\end{abstract}

\maketitle

\section{Introduction}

The theory of weight modules has its origin in Lie theory and generalizes well-known aspects of linear algebra. Much of the early study of generalized Weyl algebras (GWAs) and related algebras focused on generalizing the theory of the enveloping algebra $U(\fsl_2)$. In particular, the primitive quotients of $U(\fsl_2)$ are rank one GWAs over the polynomial ring in one variable. The formal definition of a GWA is given in Definition \ref{defn.GWA}.

In \cite{BR}, Bell and Rogalski classified $\ZZ$-graded simple rings that are birationally commutative. The canonical example of such an algebra is the Weyl algebra. For their classification, they introduced a construction that generalizes (most) rank one generalized Weyl algebras (GWAs). In this work, we study ring-theoretic and module-theoretic aspects of Bell and Rogalski's algebras, inspired by those for GWAs.

Throughout this paper, we fix $\kk$ to be a field. All rings in this paper will be associative unital $\kk$-algebras. The algebras in \cite{BR} may be realized as subalgebras of certain Ore extensions that we describe now. For an algebra $R$ and $\sigma \in \Aut_{\kk}(R)$, the \emph{skew Laurent algebra} over $R$, denoted $R[t,t\inv;\sigma]$ is the algebra generated by $t,t\inv$ over $R$ with the rule that $t^{\pm 1}r=\sigma^{\pm 1}(r)t^{\pm}$ for all $r\in R$.

\begin{definition}\label{defn.BR}
Let $R$ be an algebra with an automorphism $\sigma \in \Aut_{\kk}(R)$. Let $H$ and $J$ be two-sided ideals of $R$. The corresponding \emph{Bell--Rogalski algebra} (or \emph{BR algebra}, for short) is the algebra
\[ R(t,\sigma,H,J) = \bigoplus_{n \in \ZZ} I^{(n)} t^n \subseteq R[t,t^{-1};\sigma]\]
where $I^{(0)} = R$, $I^{(n)} = J \sigma(J) \cdots \sigma^{n-1}(J)$ for $n \geq 1$ and $I^{(n)} = \sigma^{-1}(H) \sigma^{-2}(H) \cdots \sigma^n(H)$ for $n \leq -1$. Throughout we assume that the ideals $H$ and $J$ are chosen so that $I^{(n)} \neq 0$ for all $n$.
\end{definition}

The condition $I^{(n)} \neq 0$ for all $n$ is meant to ensure that the algebra $B$ is ``truly" $\ZZ$-graded. Note that if $R$ is a domain and $H,J \neq 0$, then $\Ann_R(J)=\Ann_R(H) = 0$, and hence this condition is satisfied automatically.

When $H = R$ and $\sigma=\id_R$, then the corresponding BR algebra is just the extended Rees ring $R[Jt,t\inv]$ of the ideal $J$. More generally, if $H = R$ and $\sigma(J) \subset J$, then the corresponding BR algebra is a Zhang twist of $R[Jt,t\inv]$ \cite{Ztwist}. However, we will be most interested in the case where $H$ and $J$ are not both $\sigma$-invariant. 

Our definition is more general than the one given in \cite{BR}, where Bell and Rogalski were mainly interested in simple, birationally commutative rings. First, they require $R$ to be a commutative noetherian algebra that is a domain. Throughout most of this paper, we will assume this as well, but note that it is not strictly necessary for the definition. More significantly, for $B = R(t,\sigma,H,J)$ with $R$ commutative, let 
\begin{align}\label{eq.SB}
\css(B) := \{\bp\in\Spec(R)~|~\bp\supset HJ\} \simeq \Spec(R/(HJ)) \simeq \Spec(R/H)\cup \Spec(R/J).
\end{align}
We say $\css(B)$ is \emph{$\sigma$-lonely} if $\sigma^i(\css(B)) \cap \css(B) = \emptyset$ for all $i \neq 0$. Bell and Rogalski require $\css(B)$ to be $\sigma$-lonely, as this is necessary for the simplicity of $B$, however we generally make no such requirement. 

In Section \ref{sec.ring} we establish several basic ring-theoretic properties of BR algebras, building on the work of Bell and Rogalski. We make explicit the connection between GWAs and BR algebras in the next result.

\begin{theorem*}[Theorem \ref{thm.gwa}]
Let $R$ be a commutative domain. As $\ZZ$-graded algebras, the rank one GWAs over $R$ are exactly the BR algebras $R(t, \sigma, H, J)$ in which $H$ and $J$ are nonzero principal ideals.
\end{theorem*}

In addition, we explore how BR algebras behave with respect to some basic ring-theoretic constructions, such as passing to factor rings, invariants, and localization. 

In Section \ref{sec.weight} we study weight modules over BR algebras. The study of weight modules was an integral part of Bavula's initial study of GWAs \cite{B1}. Drozd, Guzner, and Ovsienko in \cite{DGO} fully classified simple and indecomposable weight modules for rank one GWAs with commutative base ring.

\begin{definition}\label{defn.wtmod}
Let $R$ be a commutative domain, and let $B$ be a BR algebra over $R$. We say that a left $B$-module $M$ is an $R$-\emph{weight module} (or just a \emph{weight module} when the choice of $R$ is clear) if
\[ M=\bigoplus_{\frm\in\Maxspec(R)}M_{\frm},\qquad M_{\frm}:=\{m\in M~|~\frm\cdot m=0\}, \qquad \dim_{R/\frm}(M_\frm)<\infty. \]
If $M$ is an $R$-weight module, we define its \emph{support} to be \[\Supp_R(M):=\{\frm\in\Maxspec(R)~|~M_{\frm}\neq 0\}.\] If $m\in M_{\frm}$ for some $\frm\in\Supp_R(M)$, we say that $m$ is a \emph{weight vector}. We denote by $(B,R)\wmod$ the full subcategory of (left) $B$-modules which are $R$-weight modules. 
\end{definition}

As is the case for GWAs, the structure of simple weight modules for BR algebras depends on whether the orbit of the $\sigma$-action on $\Maxspec(R)$, which the module is supported on, is finite or infinite. We make the following definition by analogy with the terminology used for (twisted) GWAs, which is essential in our classification results.

\begin{definition}\label{defn.breaks}
A maximal ideal $\frm\in\Maxspec(R)$ is called a \emph{break} for $B$ if $\sigma(\frm)\in \css(B)$.
\end{definition}

There is a $\ZZ$-action on $\Maxspec(R)$, defined by $n\cdot \frm=\sigma^n(\frm)$ for all $n\in\ZZ$, $\frm\in\Maxspec(R)$. For $\cO\in\Maxspec(R)/\ZZ$ we denote by $(B,R)\wmodO$ the full subcategory of modules $M\in (B,R)\wmod$ such that $\Supp_R(M)\subseteq \cO$. When $\cO$ is an infinite orbit, we give a complete classification of simple weight modules. Let $\beta\subseteq \cO$ be the set of breaks. If $\beta \neq \emptyset$, then set $\beta'=\beta$ unless $\beta$ contains a maximal element, in which case we set $\beta'=\beta \cup \{\infty\}$.

\begin{theorem*}[Theorem \ref{thm.inf-orb-class}]
Let $R$ be a commutative domain and let $B=R(\sigma,t,H,J)$ be a BR algebra. Let $\cO\in\Maxspec(R)/\ZZ$ be an infinite orbit and let $\beta \subseteq \cO$ be the set of breaks. If $\beta=\emptyset$, then up to isomorphism, there is a unique simple module in $(B,R)\wmodO$. If $\beta \neq \emptyset$, then the isomorphism classes of simple modules in $(B,R)\wmodO$ are parameterized by $\beta'$.
\end{theorem*}

Our second main result of Section \ref{sec.weight} gives a classification of simple weight modules in the case of finite orbits, under some additional assumption (see Hypothesis~\ref{hyp.finorb}). 

\begin{theorem*}[Theorem \ref{thm.fin-orb-class}]
Let $R$ be a commutative domain and let $B=R(\sigma,t,H,J)$ be a BR algebra. Assume that $\cO$ satisfies Hypothesis~\ref{hyp.finorb}.
\begin{enumerate}
\item If there are no breaks for $B$ on $\cO$, then the simple weight modules are parameterized by triples $(\frm,N,\theta)$ where $\frm \in \cO$, $N$ is a vector space over $R/\frm$, and $\theta\in \Aut_{R/\frm}(N)$ is an invertible linear transformation that leaves no nontrivial subspace invariant.
\item Assume there are breaks for $B$ on $\cO$. Then each simple module belongs to one of three families as described in Section \ref{sec.finorb}.
\end{enumerate}
\end{theorem*}

As BR algebras were defined very recently, there are many open questions about their structure and their representation theory. In addition to those asked in \cite{BR}, throughout this paper, we pose several open questions which arise naturally from our constructions which would make for interesting future directions of research.

\subsection*{Acknowledgments} 
R. Won was partially supported by an AMS--Simons Travel Grant and Simons Foundation grant \#961085.

\section{Ring-theoretic properties of BR algebras}
\label{sec.ring}

An algebra $A$ is \emph{$\ZZ$-graded} if there is a $\kk$-vector space decomposition $A=\bigoplus_{k \in \ZZ} A_k$ such that $A_k \cdot A_\ell \subseteq A_{k+\ell}$ for all $k,\ell \in \ZZ$. For $k \in \ZZ$, an element $a \in A_k$ is called \emph{homogeneous}. For a $\ZZ$-graded Ore domain $A$, we write $Q_{\mathrm{gr}}(A)$ for the graded quotient ring of $A$, the localization of $A$ at all nonzero homogeneous elements.  Throughout this paper, all $A$-modules will be left $A$-modules. An $A$-module $M$ is called \emph{$\ZZ$-graded} if there is a $\kk$-vector space decomposition $M = \bigoplus_{k \in \ZZ} M_k$ such that $A_m \cdot M_k  \subseteq M_{m+k}$ for all $m, k \in \ZZ$. A (left/right/two-sided) ideal of $A$ is \emph{$\ZZ$-graded} if it is generated by homogeneous elements. Throughout this paper, when unspecified, the term \emph{graded} will mean $\ZZ$-graded.

There are three primary goals of this section. First, we wish to collect several ring-theoretic properties of BR algebras. Many of these were established by Bell and Rogalski \cite{BR}. Others can be shown in an analogous way to GWAs. Secondly, we explain the connection between GWAs and BR algebras. Finally, we study how BR algebras behave under several constructions.

An automorphism $\sigma$ of a ring $R$ is \emph{inner} if there exists $a \in R$ such that $ra=a\sigma(r)$ for all $r \in R$. We denote the center of a ring $R$ by $\cnt(R)$.

\begin{lemma}\label{lem.props}
Let $R$ be an algebra and let $B=R(t,\sigma,H,J)$.
\begin{enumerate}
\item \label{gprops1} $B$ is a $\ZZ$-graded subalgebra of $R[t,t\inv;\sigma]$.
\item \label{gprops2} $B$ is generated as an algebra by $R$, $Jt$, and $\sigma\inv(H)t\inv$.
\item \label{gprops3} $B$ is a domain if and only if $R$ is a domain.
\item \label{gprops4} Suppose $R$ is a domain and no nontrivial power of $\sigma$ is inner.
    \begin{enumerate}
        \item If $|\sigma|=n<\infty$, then $\cnt(B)$ is generated by $\cnt(R)^{\grp{\sigma}}$ and $(I^{(\pm n)} \cap \cnt(R))t^{\pm n}$.
        
        \item If $\sigma$ has infinite order, then $\cnt(B) = \cnt(R)^{\grp{\sigma}}$.
    \end{enumerate}
\end{enumerate}
\end{lemma}

\begin{proof}
\eqref{gprops1} Set $B_n = I^{(n)}t^n$ and $B_m = I^{(m)}t^m$, $n,m \in \ZZ$. Then
\[ B_n B_m = I^{(n)}t^nI^{(m)}t^m = I^{(n)}\sigma^n(I^{(m)}) t^{n+m}.\]
An easy computation, though with several cases to consider, shows that $I^{(n)}\sigma^n(I^{(m)}) \subseteq I^{(n+m)}$. The $\ZZ$-grading is then clear by setting $\deg(R)=0$ and $\deg(t^{\pm 1})=\pm 1$.

\eqref{gprops2} This is clear.

\eqref{gprops3} On one hand, if $B$ is a domain then clearly so is its subalgebra $B_0=R$. On the other hand, if $R$ is a domain, then so is $R[t,t\inv;\sigma]$ and hence so is the subalgebra $B$.

\eqref{gprops4} Since $B$ is $\ZZ$-graded, then so is the center $\cnt=\cnt(B)$. Let $z \in \cnt$ be homogeneous. First suppose $\deg(z)=0$ so $z \in \cnt(R)$. Since $\sigma$ preserves $\cnt$, then $\sigma(z) \in \cnt$. Thus, for $at \in Jt$ we have $z(at) = (at)z = \sigma(z)(at)$. As $B$ is a domain, then $z=\sigma(z)$. Hence, $z \in \cnt(R)^{\grp{\sigma}}$.

Now suppose $\deg(z)=k \neq 0$. Then $z=at^k$ with $a \in I^{(k)}$. Let $r \in R$. Then $r(at^k)=(at^k)r = a\sigma^k(r)t^k$. As $B$ is a domain, then $ra=a\sigma^k(r)$. Since this holds for all $r \in R$, then this implies that $\sigma^k$ is an inner automorphism. If $|\sigma|=\infty$, this is a contradiction and the result follows. If $|\sigma|=n<\infty$, then $n \mid k$ and we have $ra=ar$ so $a \in \cnt(R)$.
\end{proof}

The result in Lemma \ref{lem.props}\eqref{gprops4} recovers a result of Kulkarni in the case that $B$ is a GWA  (see Definition \ref{defn.GWA}) and $R$ is a commutative domain \cite[Corollary 2.0.2]{kulk}. In that setting, the condition that no nontrivial power of $\sigma$ is inner is superfluous. When $R=\kk[t]$ and $\sigma$ has finite order, the GWA is finitely-generated as a module over its center. This prompts the next question.

\begin{question}
Under what conditions is $B$ module-finite over its center (and hence PI)?
\end{question}

We now restrict to the case that $R$ is a commutative noetherian domain. For $\sigma \in \Aut(R)$, we say $R$ is $\sigma$-simple if the only ideals with $\sigma(I)=I$ are $0$ and $R$. The automorphism $\sigma$ is called \emph{locally algebraic} if for any $r \in R$, the set $\{\sigma^n(r) : n \in \ZZ\}$ is contained in a finite-dimensional $\kk$-subspace of $R$.

For the next result, recall the definition of $\css(B)$ from equation \eqref{eq.SB}.

\begin{proposition}\label{prop.props}
Let $R$ be a commutative noetherian domain and let $B=R(t,\sigma,H,J)$.
\begin{enumerate}
    \item \label{props2} $B$ is simple if and only if  $R$ is $\sigma$-simple and $\css(B)$ is $\sigma$-lonely.
    \item \label{props3} Suppose that $\sigma^n(\bp) \neq \bp$ for any $\bp \in \Maxspec(R)$
    and any $n \neq 0$. Then every ideal of $B$ is graded.
    \item \label{props4} Suppose $\kk$ is algebraically closed, let $K$ be the field of fractions of $R$, and set $\trdeg(K/\kk) = d$. If $\sigma$ is locally algebraic, then $\GKdim B = d + 1$.
\end{enumerate}
\end{proposition}
\begin{proof}
\eqref{props2} Sufficiency follows from \cite[Proposition 2.18 (3)]{BR}. We will prove necessity.

Suppose that $R$ contains a $\sigma$-invariant ideal $I \neq 0,R$. Then $BI=IB$ is a proper, nontrivial ideal of $B$. Hence, $B$ is not simple. 

Now suppose that $\css(B)$ is not $\sigma$-lonely. Then there exists $k \neq 0$ such that $\sigma^k(\css(B)) \cap \css(B) \neq \emptyset$. That is, there is some $\bp,\bq\in \css(B)$ such that $\sigma^k(\bp)=\bq$. Define an $R$-module with generating set $v_0,\ldots,v_{k-1}$ as follows:
\[M=\bigoplus_{i=0}^{k-1}\left(R/\sigma^i(\bp)\right)v_{i}.\]
We define a $B$-action on $M$ by setting
\[
jt\cdot v_i=
\begin{cases} j v_{i+1} & \text{ if }0\leq i<k-1 \\ 0 & \text{ if }i=k-1, 
\end{cases}\qquad 
\sigma\inv(h)t\inv\cdot v_i=
\begin{cases} 
\sigma\inv(h) v_{i-1} & \text{ if }0< i\leq k-1 \\ 
0 & \text{ if }i=0,
\end{cases}\]
for all $j \in J$ and $h \in H$.
We verify that this action is well defined.

If $r\in\sigma^i(\bp)$, then $r\cdot v_i=0$ and for all $j\in J$ we have
\[ (jt)r\cdot v_i
=\sigma(r)jt\cdot v_i
=\sigma(r)jv_{i+1}
=j\sigma(r)v_{i+1}=0\]
because $\sigma(r)\in\sigma^{i+1}(\bp)$. Similarly, for all $h\in H$ we have
\[ (\sigma\inv(h)t\inv)r\cdot v_i
=\sigma\inv(r)\sigma\inv(h)t\inv\cdot v_i
=\sigma\inv(r)\sigma\inv(h)v_{i-1}
=\sigma\inv(h)\sigma\inv(r)v_{i-1}=0\]
because $\sigma\inv(r)\in\sigma^{i-1}(\bp)$.

Let $j\in J$ and $h\in H$. Then we have $\sigma\inv(h)t\inv jt=\sigma\inv(hj)\in R$ and if $0\leq i<k-1$,
\[\sigma\inv(h)t\inv jt\cdot v_i=\sigma\inv(h)t\inv j v_{i+1}=\sigma\inv(j)\sigma\inv(h)t\inv\cdot v_{i+1}=\sigma\inv(j)\sigma\inv(h)v_i=\sigma\inv(hj)v_i. \]
If $i=k-1$, we get that 
\[0=\sigma\inv(h)t\inv\cdot 0=\sigma\inv(h)t\inv (jt\cdot v_{k-1})=\sigma\inv(hj)\cdot v_{k-1} \]
which is satisfied because $\bq=\sigma^k(\bp)\in\css(B)$, hence $\sigma\inv(hj)\in\sigma\inv(HJ)\subseteq \sigma\inv(\bq)=\sigma^{k-1}(\bp)$. Analogously, for all $j\in J$, $h\in H$ we have $ jt\sigma\inv(h)t\inv=jh=hj\in R$ and if $0< i\leq k-1$,
\[jt\sigma\inv(h)t\inv \cdot v_i=jt \sigma\inv(h) v_{i-1}=hj t \cdot v_{i-1}=hj v_i. \]
If $i=0$, then
\[0=jt\cdot 0=jt(\sigma\inv(h)t\inv\cdot v_{0})=jh\cdot v_{0} \]
which is satisfied because $\bp\in\css(B)$, hence $jh=hj\in HJ\subseteq \bp$.

Notice that $M$ is a graded module, by letting $M_i=(R/\sigma^i(\bp))v_i$, $i=0,\ldots,k-1$. Then $\Ann_B(M)\subseteq B$ is a nontrivial two sided graded ideal of $B$ since $1\not\in\Ann_B(M)$, but $B_{\geq k},B_{\leq -k}\subseteq \Ann_B(M)$. This implies that $B$ is not simple.

\eqref{props3} We follow \cite{bavsimp}. In particular, \cite[Lemma 5]{bavsimp} implies immediately that for $n \in \ZZ$, $n \neq 0$,
\begin{align}
\label{eq.Dsum}
\sum_{u \in R} R(u-\sigma^n(u)) = R.
\end{align}

Suppose $M$ is a $R$-subbimodule of $B$ and let $v \in M$. Write $v=v_n + \cdots + v_m$ where $v_i \in B_i$, $n \leq m$. 
We claim $v_i \in M$ for all $i$. Let $\ell(v)=m-n$ denote the length of $v$. If $\ell(v)=0$, then there is nothing to prove, so assume the result is true for all elements in $M$ with length less than $k>0$. By \eqref{eq.Dsum}, there exist elements $\alpha_i, \beta_i \in R$ such that 
\[ \sum \alpha_i (\sigma^n(\beta_i)-\sigma^m(\beta_i)) = 1.\]
Set
\[ u := \sum \alpha_i(v\beta_i - \sigma^m(\beta_i)v).\]
Then $u \in M$ and write $u=u_n + \cdots + u_m$ where $u_n=v_n$ and $u_m=0$. Thus, $\ell(u) < \ell(v)$ so by the induction hypothesis, $v_n \in M$. Hence, $v_i \in M$ for all $i$.

\eqref{props4}
As noted in \cite{BR}, it is easy to show $\GKdim B \geq d + 1$. Let $V \subseteq K$ be the span of a transcendence basis for $K/\kk$ and let $a \in B$ be a nonzero element of positive degree. Then consider the subalgebra $\kk\langle V + \kk a\rangle$.

On the other hand, if $\trdeg(K/\kk) = d$ then $R$ has Krull dimension $d$ and, as $R$ is commutative, $\GKdim R = d$. Since $\sigma$ is locally algebraic, then $\GKdim R[t, t\inv; \sigma] = d + 1$ \cite[Proposition 1]{LMO}. As $B$ is a subalgebra of $R[t, t\inv; \sigma]$, then $\GKdim B \leq d+1$ and the result follows.
\end{proof}

A collection of distinct closed subsets $\{W_\alpha\}$ of a variety $\Spec R$ is \emph{critically dense} if for any proper closed subset $C \subseteq \Spec R$, one has $C \cap W_\alpha = \emptyset$ for all but finitely many $\alpha$. Bell and Rogalski proved that if $R$ is a commutative noetherian domain and $\{\sigma^i(\css(B)) : i \in \ZZ \}$ is critically dense in $\Spec R$, then $B=R(t,\sigma,H,J)$ is noetherian. This is related to \cite[Prop 3.10]{KRS}.

However, the critical density condition is not necessary for $B$ to be noetherian, as we now show. Consider $R=\CC[z_1,z_2]$, $\sigma(z_i)=z_i-1$, $J=(z_1)$, $H=R$. Since $R$ is commutative and $J, H$ are principal, then by Theorem \ref{thm.gwa} below, $B$ is a generalized Weyl algebra over a noetherian base ring, hence it is noetherian \cite[Proposition 1.3]{B1}.
(Notice that unlike the setting of Bell and Rogalski's proof, here $B$ is not simple because $(z_1-z_2)$ is a $\sigma$-invariant ideal, so $R$ is not $\sigma$-simple). We have 
\[ \css(B) = \{\bp\in\Spec(R)~|~\bp\supset HJ\}=\{\bp\in\Spec(\CC[z_1,z_2])~|~\bp\supset (z_1)\}=\{(z_1,z_2-\alpha)~|~\alpha\in\CC\}\cup\{(z_1)\}\]
therefore, for all $i\in\ZZ$, $\sigma^i(\css(B))=\{(z_1-i,z_2-\alpha)~|~\alpha\in\CC\}\cup\{(z_1-i)\}$. Now take $C=\{\bp\in\Spec(R)~|~\bp\supset (z_2)\}$ which is a proper closed subset of $\Spec(R)$, then, for all $i\in \ZZ$, we have
\[C\cap \sigma^i(\css(B))=\{(z_1-i,z_2)\}\neq\emptyset.\]

\begin{question}
Is there a condition on $R$, $\sigma$, $H$, $J$ that is equivalent to $B$ being noetherian?
\end{question}

\subsection{BR algebras and GWAs}

We next move to a discussion of generalized Weyl algebras and how BR algebras generalize this notion.

\begin{definition}\label{defn.GWA}
Let $R$ be a unital $\kk$-algebra. Let $\sigma \in \Aut(R)$ and let $a \in \cnt(R)$ be a non-zero divisor. The \emph{(rank one) generalized Weyl algebra} (GWA) is generated as an algebra over $R$ by $x$ and $y$ subject to the relations
\[ xr = \sigma(r)x, \quad yr = \sigma\inv(r)y, \quad
yx = a, \quad xy=\sigma(a),\]
for all $r \in R$. We denote the GWA above by $R(x,y,\sigma,a)$.
\end{definition}

The GWA $R(x,y,\sigma, a)$ is naturally $\ZZ$-graded by setting $\deg x = 1$, $\deg y = -1$, and $\deg r = 0$ for all $r \in R$. We remark that in some definitions, see e.g. \cite{B1}, the element $a$ in the above definition is not required to be a non-zero divisor. In particular, we could even have $a=0$. This leads to some degenerate cases which cannot be expressed as a BR algebra as we have defined them. In \cite{BR} it is shown that any simple GWA may be realized as a BR algebra. In fact, this is true much more generally.

\begin{theorem}\label{thm.gwa}
Let $R$ be a commutative domain. As $\ZZ$-graded algebras, the rank one GWAs over $R$ are exactly the BR algebras $R(t, \sigma, H, J)$ in which $H$ and $J$ are nonzero principal ideals.
\end{theorem}
\begin{proof}
Let $J = (j)$ and $H = (h)$. Let  $a=\sigma\inv(jh)$ and set $A = R(x,y,\sigma,a)$. We claim $A \iso B$.

Define a map $\phi:A \sto B$ by 
\[ 
r \mapsto r \text{ for all $r \in R$}, \quad 
x \mapsto jt, \quad 
y \mapsto \sigma\inv(h) t\inv.\]
Then
\begin{align*}
    \phi(x)\phi(r)-\phi(\sigma(r))\phi(x) 
        &= (jt)r-\sigma(r)(jt) = j(tr-\sigma(r)t) = 0 \\
    \phi(y)\phi(r)-\phi(\sigma\inv(r))\phi(y) 
        &= (\sigma\inv(h) t\inv)r-\sigma\inv(r)(\sigma\inv(h) t\inv) 
        = \sigma\inv(h) (t\inv r - \sigma\inv(r) t\inv) = 0 \\ 
    \phi(y)\phi(x) - \phi(a) 
        &= (\sigma\inv(h) t\inv)(jt) - a 
        = \sigma\inv(h)\sigma\inv(j) - a = 0 \\
    \phi(x)\phi(y) - \phi(\sigma(a))
        &= (jt)(\sigma\inv(h) t\inv) - \sigma(a)
        = jh - \sigma(a) = 0.
\end{align*}
Hence, $\phi$ extends to a graded algebra homomorphism $A \sto B$. 

By Lemma~\ref{lem.props}\eqref{gprops2}, $B$ is generated as an algebra by $R$, $Jt$ and $\sigma\inv(H)t\inv$ and so it is generated as an algebra by $R$, $jt$ and $\sigma\inv(h)t\inv$. Therefore, we can define a graded algebra map $\psi:B \sto A$ by
\[
r \mapsto r \text{ for all $r \in R$}, \quad jt \mapsto x, \quad \sigma\inv(h)t\inv \mapsto y.
\]
Since $B$ is a subalgebra of the skew Laurent ring $R[t, t\inv; \sigma]$, we need only verify that $\psi$ preserves the relations in the skew Laurent ring (between the generators of $B$). Note that for all $r \in R$, we have
\begin{align*}
\psi(jt)\psi(r) - \psi(\sigma(r))\psi(jt)
    &=  xr - \sigma(r)x = 0\\
\psi(\sigma\inv(h)t\inv)\psi(r) - \psi(\sigma\inv(r))\psi(\sigma\inv(h)t\inv)
    &=  yr - \sigma\inv(r)y = 0\\
\psi(jt)\psi(\sigma\inv(h)t\inv) - \psi(jh)
    &=  xy - jh = xy - \sigma(a) = 0\\
\psi(\sigma\inv(h)t\inv)\psi(jt) - \psi(\sigma\inv(hj)) &= yx - \sigma\inv(jh) = yx - a = 0.
\end{align*}
Hence, $\psi$ extends to a graded algebra homomorphism $B \sto A$. It is clear that $\phi$ and $\psi$ are mutually inverse.

Conversely, if we fix $a \in R$, then the same argument as above shows that $R(x,y,\sigma,a)$ is a BR algebra with $J=(\sigma(a))$ and $H=R$.
\end{proof}

\begin{corollary}
If $R$ is a PID with nonzero ideals $H$ and $J$, then the BR algebra $R(t,\sigma,H,J)$ is a GWA. 
\end{corollary}

There are many BR algebras that cannot be realized as a GWA. We give one example below.

\begin{example}\label{ex.nongwa}
Let $R=\kk[z_1,z_2]$ and define $\sigma \in \Aut_\kk(R)$ by $\sigma(z_1)=z_1-1$ and $\sigma(z_2)=z_2+1$. Set $H=(z_1-1,z_2+1)R$, $J=R$. Then $B=R(t,\sigma,H,J)$ is a BR algebra that is not simple because the ideal $(z_1+z_2)$ is fixed by $\sigma$ and hence $R$ is not $\sigma$-simple. It is clear that $B$ is not isomorphic to a GWA \emph{as graded algebras} because $B_{-1}$ is not a free module of rank $1$ over $B_0$. We conjecture that $B$ is not algebra-isomorphic to a GWA. 
\end{example}

\begin{remark}
Let $A=R(x,y,\sigma,a)$ be a GWA over a commutative base ring $R$. Then $A$ is the BR algebra $R(t,\sigma,H,J)$ with $J=(a)$ and $H=R$. Thus, $\css(B)=\{ \bp \in \Spec(R) \mid a \in \bp\}$. Suppose $A$ is simple. By Proposition \ref{prop.props}\eqref{props2}, $\css(B)$ is $\sigma$-lonely and this is equivalent to the condition that $\sigma$ has infinite order and for every positive integer $n$, $R=aR+\sigma^n(a)R$. Thus, the conditions for simplicity in a BR algebra reduce to those for a GWA \cite[Theorem 6.1]{Jprim}.
\end{remark}

Suppose $H=(h)$, $J=(j)$, $H'=(h')$, and $J'=(j')$ are nonzero principal ideals in a commutative domain $R$ and let $\sigma \in \Aut(R)$. By Theorem \ref{thm.gwa}, the BR algebras $B=R(t,\sigma,H,J)$, $B'=R(t,\sigma,H',J')$ are GWAs. If $hj=h'j'$, then it follows that $B \iso B'$. The next result generalizes this notion.

\begin{proposition}\label{prop.HJ-split}
Let $R$ be a commutative algebra and let $B=R(t,\sigma,H,J)$ and $B'=R(t,\sigma,H',J')$. Suppose that $\varphi_H:H\sto H'$ and $\varphi_J:J\sto J'$ are isomorphisms of $R$-modules such that
\[ \varphi_H(h)\varphi_J(j)=hj\quad \text{ for all }h\in H,~j\in J,\]
then $B\simeq B'$ as graded algebras.
\end{proposition}
\begin{proof}
The map $\varphi:B\sto B'$ such that $\varphi|_{B_0}=\id_R$, $\varphi(jt)=\varphi_J(j)t$, $\varphi(\sigma\inv(h)t^{-1})=\sigma\inv(\varphi_H(h))t\inv$, extended multiplicatively is the required graded isomorphism.
Note that the hypotheses imply, in particular, that $HJ = H'J'$.
\end{proof}

\begin{question}
In general, the isomorphism problem for BR algebras is open. For a fixed base ring $R$ and automorphism $\sigma$, which pairs of ideals $(H,J)$ and $(H', J')$ yield isomorphic BR algebras?
\end{question}

\subsection{Constructions}

We now study maps between various BR algebras and consider various ring-theoretic constructions, including invariant rings, quotients, and localizations. These all rely heavily on the following functoriality result for BR algebras.

\begin{proposition}\label{prop.induced}
Let $B=R(t,\sigma,H,J)$ and $B'=R'(t',\sigma',H',J')$. Suppose $\phi:R \sto R'$ is a $\kk$-algebra homomorphism such that $\phi(H) \subseteq H'$, $\phi(J)\subseteq J'$, and $\phi\sigma=\sigma'\phi$. Then for each $\gamma \in \kk^{\times}$ there is a graded algebra morphism $\Phi_{\gamma} :B \sto B'$ satisfying $\Phi_{\gamma}(r t) = \phi(r) \gamma t'$ for all $r \in J \cup \sigma\inv(H)$ such that the following diagram commutes:
\[\begin{tikzcd}        
    R \arrow{r}{\phi} \arrow{d}{\iota} & R'  \arrow{d}{\iota'} \\
    B \arrow{r}{\Phi_\gamma} & B' 
\end{tikzcd}\]
Furthermore, we have the following:
\begin{enumerate}
    \item If $\Phi_{\gamma}$ is injective (resp. surjective), then $\phi$ is injective (resp. surjective).
    \item If $R'$ is a domain and $\phi$ is injective, then $\Phi_{\gamma}$ is injective.
    \item If $\phi(H)=H'$, $\phi(J)=J'$, and $\phi$ is surjective, then $\Phi_{\gamma}$ is surjective.
\end{enumerate}
\end{proposition}
\begin{proof}
Define a map $\Phi_{\gamma}: R[t,t\inv;\sigma] \sto R'[t',(t')\inv;\sigma']$ by $\Phi_{\gamma}(t) = \gamma t'$, $\Phi_{\gamma}(t\inv) = \gamma\inv (t')\inv$, and $\Phi_{\gamma}(r) = \phi(r)$ for all $r \in R$. 
Since
\[
\Phi_{\gamma}(tr - \sigma(r)t) =
\Phi_{\gamma}(t)\Phi_{\gamma}(r)-\Phi_{\gamma}(\sigma(r))\Phi_{\gamma}(t)
    = \gamma t'\phi(r) - \phi(\sigma(r)) \gamma t'
    = \gamma (\sigma'\phi-\phi\sigma)(r)t'
    = 0,
\]
Then $\Phi_{\gamma}$ extends to an algebra homomorphism $R[t,t\inv;\sigma] \sto R'[t',(t')\inv;\sigma']$. Since $\phi(H) \subseteq H$ and $\phi(J) \subseteq J$, then it is clear that $\Phi_{\gamma}$ restricts to a homomorphism $B \sto B'$. By abuse of notation we call the restriction $\Phi_{\gamma}$, as well. By construction, $\Phi_{\gamma}$ respects the $\ZZ$-grading, so is a graded algebra map.

(1) Since $\Phi_\gamma$ and $\phi$ agree on $R$, then it is clear that when $\Phi_\gamma$ injective/surjective, then so is $\phi$. 

(2) Suppose $R'$ is a domain and $\phi$ is injective. Let $at^n \in \ker\Phi_\gamma$. Then $0=\Phi_\gamma(at^n) = \gamma^n \phi(a) (t')^n$. Since $B'$ is a domain by Lemma \ref{lem.props}\eqref{gprops3}, then $\phi(a)=0$. Thus, $a=0$.

(3) Suppose $\phi(H)=H'$, $\phi(J)=J'$, and $\phi$ is surjective. To show that $\Phi_{\gamma}$ is surjective, it suffices to show that the image contains $B'_0$, $B'_{1}$ and $B'_{-1}$, as these generate $B'$ as an algebra. If $a' \in B'_0 = R'$, since $\phi$ is surjective, there exists $a \in B_0$ such that $a' = \phi(a) = \Phi_{\gamma}(a)$. Further, if $a't' \in B'_1$, then $a' \in J$ and so by hypothesis, there exists $a \in J$ such that $\phi(a) = a'$. Then $\Phi_{\gamma}(\gamma\inv at) = a't'$. Analogously, $B'_{-1}$ is in the image of $\Phi_{\gamma}$, and so $\Phi_{\gamma}$ is surjective.
\end{proof}

The previous result implies that if $\phi$ is an automorphism of $R$ such that $\phi$ commutes with $\sigma$, $\phi(H) \subseteq H$ and $\phi(J)\subseteq J$, then for every $\gamma \in \kk^\times$, $\phi$ lifts to a graded automorphism $\Phi_\gamma$ of $R(t,\sigma,H,J)$. The following is a generalized version of a result of Jordan and Wells \cite{JW} for rank one GWAs. Invariant rings of GWAs have also been studied in \cite{GHo,GR,GW1,KK}.

\begin{proposition}\label{prop.fixed}
Let $R$ be a commutative noetherian domain and let $B=R(t,\sigma,H,J)$. Let $\gamma \in \kk^\times$ and let $\phi \in \Aut(R)$ satisfying the conditions of Proposition \ref{prop.induced} for $B=B'$, $H=H'$, and $J=J'$. Then $\phi$ lifts to an automorphism $\Phi_\gamma$ of $B$. Set $n=\ord(\phi)$, $m=\ord(\gamma)$, and
\begin{align*}
    K &= \left(\sigma\inv(H)\sigma^{-2}(H) \cdots \sigma^{-m}(H) \right) \cap R^\grp{\phi} \\
    L &= \left( J\sigma(J) \cdots \sigma^{m-1}(J) \right) \cap R^\grp{\phi}.
\end{align*}
If $n,m < \infty$ and $\gcd(n,m)=1$, then
\[ B^\grp{\Phi_\gamma} = R^\grp{\phi}\left(t^m,\sigma^m,K,L \right).\]
\end{proposition}
\begin{proof}
It is clear that $R^\grp{\phi}\left(t^m,\sigma^m,K,L \right) \subset B^\grp{\Phi_\gamma}$. We will show the other inclusion.

Since $\Phi_\gamma$ respects the $\ZZ$-grading on $B$, then it follows that $B^\grp{\Phi_\gamma}$ is again $\ZZ$-graded. It is clear that $(B^\grp{\Phi_\gamma})_0=R^\grp{\phi}$.

Let $k \neq 0$ and suppose $at^k \in I^{(k)}t^k$ is fixed by $\Phi_\gamma$. Since $at^k = \Phi_\gamma(at^k) = \gamma^k\phi(a) t^k$, then because $R$ (and hence $B$) is a domain, then we have $\phi(a) = \gamma^{-k}a$. But then $a=\phi^m(a) = (\gamma^{-k})^m a = (\gamma^m)^{-k}a$. Hence, $\left|\orb_\phi(a)\right|$ divides $m$, but $\left|\orb_\phi(a)\right|$ divides $n$ by the Orbit-Stabilizer Theorem. This is a contradiction unless $k \mid m$. Hence, $k=\ell m$ and so $at^k = \phi(at^k) = \phi(a)t^k$, so $a \in R^\grp{\phi}$.
\end{proof}

The gcd condition is necessary in the proof above, as illustrated by the following example.

\begin{example}
Set $R=\kk[z]$ and define $\sigma \in \Aut(R)$ by $\sigma(z)=qz$ for some $q \in \kk^\times$. Fix $a \in \kk[z^2]$, then set $J=(a)$ and $H=R$. The corresponding BR algebra $B=R(t,\sigma,H,J)$ is isomorphic to the GWA $\kk[z](\sigma,a)$ by Theorem \ref{thm.gwa}.

Define $\phi \in \Aut(B)$ to be the automorphism given by $\phi(z)=-z$ and $\phi(t)=-t$. Note that $\phi$ respects the $\ZZ$-grading on $B$ and hence $(B^\grp{\phi})_k = (B_k)^\grp{\phi}$ for all $k$. Clearly $(B^\grp{\phi})_0=(B_0)^\phi = \kk[z^2]$. Now $(B^\grp{\phi})_1$ is the $(B^\grp{\phi})_0$-module generated by $zt$. But $(B^\grp{\phi})_2$ is generated over $(B^\grp{\phi})_0$ by $t^2$. Hence, $(B^\grp{\phi})_{>0}$ is not generated as an algebra by $(B^\grp{\phi})_1$ and so it is not a BR algebra.
\end{example}

\begin{proposition}\label{prop.quot}
Let $R$ be a commutative algebra. Let $B=R(t,\sigma,H,J)$ and let $L$ be a $\sigma$-stable prime ideal of $R$ such that $L \notin \css(B)$ (see \eqref{eq.SB}). Let $\overline{\sigma}$ denote the induced automorphism on $\overline{R}=R/L$, let $\overline{H}=(H+L)/L$, and let $\overline{J}=(J+L)/L$. Finally, set $\overline{B}=\overline{R}(\overline{t},\overline{\sigma},\overline{H},\overline{J})$. Then $\overline{B} \iso B/K$ where $K=\bigoplus_{n\in\ZZ} (L\cap I^{(n)})t^n$.
\end{proposition}
\begin{proof}
Let $\phi:R \sto \overline{R}$ be the canonical map. It is clear that $\phi(H)=\overline{H}$ and $\phi(J)=\overline{J}$. Setting $\gamma=1$, Proposition \ref{prop.induced} now implies there exists a surjective map $\Phi:B \sto \overline{B}$. Clearly $K \subset \ker\Phi$. It remains to show that $\ker\Phi \subset K$.

Let $b \in \ker\Phi$. Since $\Phi$ is a homogeneous map of degree 0, we may assume $b$ is homogeneous. If $\deg(b)=0$, then $b\in R$ and $0=\Phi(b)=b+L \in \overline{R}$, so $b \in L$. Now suppose $\deg(b)=n \neq 0$. Then $b=at^n$ for some $a \in I^{(n)}$ so $0=\Phi(b) = \phi(a)\overline{t}^n$. Hence, $0 = \phi(a) = \overline{a}$ and it follows that $a \in L$.
\end{proof}
\begin{example}
(1) Let $R=\kk[z_1,z_2]$, $J=(z_1)$, $H=R$, and define $\sigma \in \Aut(R)$ by $\sigma(z_1)=qz_1$ for some $q \in \kk^\times$ and $\sigma(z_2)=z_2$. Let $B=R(t,\sigma,H,J)$ and let $L=(z_2)$ in $R$. Then $L$ is $\sigma$-stable and $L \notin \css(B)$. Let $\Phi$ be as in Proposition \ref{prop.quot} and suppose $b = at \in \ker\Phi$ with $a \in J$. Then $\phi(a)=0$ so $a \in L$. Thus, $a \in J \cap L = JL$. It follows in this example that $\ker\Phi = BL$. 

(2) Let $R=\kk[z_1,z_2,z_3]$, $J=(z_1,z_3)$, $H=R$, and define $\sigma \in \Aut(R)$ by $\sigma(z_1)=qz_1$ for some $q \in \kk^\times$, $\sigma(z_2)=z_2$, and $\sigma(z_3)=z_3$. Let $B=R(t,\sigma,H,J)$ and let $L=(z_2,z_3)$ in $R$. Then $L$ is $\sigma$-stable and $L \notin \css(B)$. Let $\Phi$ be as in Proposition \ref{prop.quot} and suppose $b = at \in \ker\Phi$ with $a \in J$. Then $\phi(a)=0$ so $a \in L$. Thus, $a \in J \cap L$ but in this case $J \cap L = (z_1z_2,z_3) \neq JL$. So in this case $\ker\Phi \neq BL$.
\end{example}

\begin{lemma}\label{lem.local}
Let $R$ be a commutative noetherian domain and set $B=R(t,\sigma,H,J)$. Let $S$ be a $\sigma$-stable multiplicative subset of $R$. Then $\sigma$ extends to a unique automorphism $\widehat{\sigma}: S\inv R \to S\inv R$. Letting $\widehat{H} = S\inv H$ and $\widehat{J} = S\inv J$, define $\widehat{B} = S\inv R(\widehat{t}, \widehat{\sigma}, \widehat{H}, \widehat{J})$. Then, viewing $S$ as a subset of $B$, we have $S\inv B = \widehat{B}$.
\end{lemma}
\begin{proof}
This is similar to the proof of Proposition \ref{prop.quot}. Let $\phi:R \sto S\inv R$ be the canonical map. Clearly $\phi(H) \subseteq \widehat{H}$ and $\phi(J) \subseteq \hat{J}$. Setting $\gamma=1$, Proposition \ref{prop.induced} implies the existence of a map $\Phi:B \sto \widehat{B}$. The map $\phi$ is injective and hence so is $\Phi$. Since $\Phi$ maps each element of $S$ to a unit in $\widehat{B}$, then by the universal property of localization there is a map,
\[ \Psi: S\inv B \to \widehat{B}.\]
It is clear that $\Psi$ is surjective. Let $a \in \ker\Psi$. Then $sa \in B$ for some $s \in S$.  The restriction of $\Psi$ to $B$ coincides with $\Phi$, and hence $\Phi(sa)=\Psi(sa)=\Psi(s)\Psi(a)=0$. Since $\Phi$ is injective, this implies that $sa=0$, so $a=0$.  
\end{proof}

The hypotheses in the previous lemma can be relaxed somewhat. In particular, the condition that $R$ be a commutative domain can be replaced by the condition that $S$ consist of regular central elements. We leave the details for the interested reader.

Under the hypotheses of Lemma \ref{lem.local}, we can take $S=R\backslash\{0\}$, which is clearly an Ore set of $R$ (and hence of $B$). Then $S\inv B = \Frac(R)(\widehat{t},\widehat{\sigma},\widehat{H},\widehat{J}) \iso \Frac(R)[t,t\inv;\sigma]$.

\section{Weight modules}\label{sec.weight}

For this section, we let $B=R(t,\sigma,H,J)$ be a BR algebra, where $R$ is a $\kk$-algebra which is a commutative domain. We would like to describe all the simple weight modules, i.e. the simple objects in $(B,R)\wmod$. There is a $\ZZ$-action on $\Maxspec(R)$, defined by $n\cdot \frm=\sigma^n(\frm)$ for all $n\in\ZZ$, $\frm\in\Maxspec(R)$. We frequently make use of the set $\css(B)$ as defined in \eqref{eq.SB}. 

\begin{notation}
For $\cO\in\Maxspec(R)/\ZZ$ we denote by $(B,R)\wmodO$ the full subcategory of modules $M\in (B,R)\wmod$ such that $\Supp_R(M)\subseteq \cO$.
\end{notation}

Notice that if $M\in (B,R)\wmod$, then $B_{\pm 1}M_{\frm} \subseteq M_{\sigma^{\pm 1}(\frm)}$ for each $\frm \in \Supp_R(M)$, so it follows immediately that
\[(B,R)\wmod\simeq \prod_{\cO\in\Maxspec(R)/\ZZ}(B,R)\wmodO.\]

In light of the above, our goal will be to describe the simple objects in $(B,R)\wmodO$ for the various orbits $\cO \in \Maxspec(R)/\ZZ$. This analysis breaks down into two cases depending on whether $\cO$ is infinite or finite. For the former case, in Section~\ref{sec.inforb} below, we give a complete description of the simple weight modules. The finite-orbit case will require some additional hypotheses and is discussed in Section~\ref{sec.finorb} below.

If $\frm\in\Maxspec(R)$, then $\sigma$ induces a field isomorphism $R/\frm\stackrel{\simeq}{\sto} R/\sigma(\frm)$, hence we can consider both $M_\frm$ and $M_{\sigma(\frm)}$ as vector spaces over $R/\frm$ (or $R/\sigma(\frm)$, if desired). The elements $j \in J$ and $h \in H$ define $R/\frm$-linear maps:
\begin{align}\label{eq.mmaps}
jt:M_{\frm}\sto M_{\sigma(\frm)},\qquad \sigma\inv(h)t\inv: M_{\sigma(\frm)}\sto M_{\frm},
\end{align}
given by left multiplication by $jt$ and $\sigma\inv(h)t\inv$, respectively.

\begin{proposition}\label{prop.inj-maps}
Let $M\in (B,R)\wmod$ and $\frm\in\Maxspec(R)$ such that $\sigma(\frm)\not\in\css(B)$. 
\begin{enumerate}
\item \label{inj1} For all $j\in J\setminus\sigma(\frm)$ and $h\in H\setminus\sigma(\frm)$ the maps \eqref{eq.mmaps} of $R/\frm$-vector spaces are injective.
\item \label{inj2} We have 
$\dim_{R/\frm}M_\frm=\dim_{R/\sigma(\frm)}M_{\sigma(\frm)}$.
\end{enumerate}
\end{proposition}
\begin{proof}
\eqref{inj1} Let $j\in J\setminus\sigma(\frm)$, and let $w\in \ker(jt)$, so $(jt)w=0$. Since $\sigma(\frm)\not\in\css(B)$, $HJ\not\subseteq \sigma(\frm)$, so $H\not\subseteq\sigma(\frm)$ and there exists $h\in H\setminus\sigma(\frm)$. Then
\[ 0=(jt)w=\sigma\inv(h)t\inv (jt)w=\sigma\inv(hj)w.\]
Since $\sigma\inv(hj)\not\in\frm$, it is invertible in $R/\frm$, hence this implies that $w=0$ and the map $jt$ is injective. The argument for the map $\sigma\inv(h)t\inv$ is entirely analogous.

\eqref{inj2} Since $\sigma(\frm)\not\in\css(B)$, $H\not\subseteq \sigma(\frm)$ and $J\not\subseteq\sigma(\frm)$, hence there exist $j\in J\setminus\sigma(\frm)$, $h\in H\setminus\sigma(\frm)$. Then by \eqref{inj1}, the maps $jt$ and $\sigma\inv(h)t\inv$ are both injective, so the two vector spaces have the same dimension.
\end{proof}

\begin{corollary}\label{lem.breaks}
Let $M\in (B,R)\wmod$ and let $0\neq w\in M_{\frm}$.
\begin{itemize}
    \item If $B_1\cdot w=0$, then $\sigma(\frm)\in\css(B)$.
    \item If $B_{-1}\cdot w=0$, then $\frm\in\css(B)$.
\end{itemize} 
\end{corollary}
\begin{proof}
This follows by contrapositive from Proposition \ref{prop.inj-maps}\eqref{inj1}.
\end{proof}

\subsection{The infinite orbit case}\label{sec.inforb}

In this section we show that, when working with modules supported on infinite orbits, the simple weight modules for BR algebras have a similar structure to the ones for GWAs. In particular the simple modules are determined by their support, which is an interval between breaks within the orbit, and all the weight spaces are one-dimensional.

\begin{proposition}\label{prop.dim-weight}
Suppose that $\cO=\ZZ\cdot \frm \subset \Maxspec(R)$ is an infinite orbit and $M\in (B,R)\wmodO$.
\begin{enumerate}
    \item \label{dim1} If $w\in M_{\frm}$ is a weight vector and $N=B\cdot w\subseteq M$, then $\dim_{R/\sigma^k(\frm)}N_{\sigma^k(\frm)}\leq 1$ for all $k\in \ZZ$.
    \item \label{dim2} If $M$ is simple, then 
    $\dim_{R/\sigma^k(\frm)}M_{\sigma^k(\frm)}\leq 1$ for all $k\in \ZZ$.
\end{enumerate}
\end{proposition}
\begin{proof}
\eqref{dim1}
If $w=0$, then $B\cdot w=0$, so the statement is true. If $w\neq 0$, since $w\in M_{\frm}$, then $\frm \cdot w=0$, so $B_0\cdot w=R\cdot w=(R/\frm)\cdot w$. Notice that for all $k\in\ZZ$, we have $B_k w\subseteq N_{\sigma^k(\frm)}$ and, since $\cO$ is an infinite orbit, $\sigma^k(\frm)\neq \sigma^{n}(\frm)$ for $k\neq n$. Then $B_k w=N_{\sigma^k(\frm)}$ for all $k\in\ZZ$ and $B_k w\cap B_n w=0$ for $k\neq n$.
It follows that 
\[N_\frm=(B w)_\frm=(B_0 w)_\frm\simeq (R/\frm) w\]
is one-dimensional over $R/\frm$. So the statement is proved for $k=0$. Now we proceed by induction in both directions. Suppose that the statement is true for $k\in \ZZ$. If $\dim_{R/\sigma^k(\frm)}N_{\sigma^k(\frm)}=0$, then $N_{\sigma^k(\frm)}=0$ and for all $n\geq k$, we have $N_{\sigma^n(\frm)}=B_{n-k}N_{\sigma^k(\frm)}=0$. If $\dim_{R/\sigma^k(\frm)}N_{\sigma^k(\frm)}=1$, let $N_{\sigma^k(\frm)}=R/\sigma^k(\frm)\cdot u$. If $N_{\sigma^{k+1}(\frm)}=B_1N_{\sigma^{k}(\frm)}=0$, then the statement is true, otherwise let $u_1, u'_1\in N_{\sigma^{k+1}(\frm)}=B_1N_{\sigma^{k}(\frm)}$, then there exist $j,j'\in J$ be such that $u_1:=jt\cdot u$, $u'_1:=j't\cdot u$.
It follows that
\[j' u_1-ju'_1= j'jt u-jj'tu=0,\]
hence $u'_1$ and $u_1$ are linearly dependent over $R$, and since $u'_1, u_1\in N_{\sigma^{k+1}(\frm)}$, it descends to being linearly dependent over $R/\sigma^{k+1}(\frm)$. 

This shows that in all cases $\dim_{R/\sigma^{k+1}(\frm)}(N_{\sigma^{k+1}(\frm)})\leq 1$. By induction, $\dim_{R/\sigma^{k}(\frm)}(N_{\sigma^{k}(\frm)})\leq 1$ for all $k\geq 0$. With an analogous reasoning, using the action of $B_{-1}$, $\dim_{R/\sigma^{k}(\frm)}(N_{\sigma^{k}(\frm)})\leq 1$ for all $k\leq 0$, and the result is proved.

\eqref{dim2} For any nonzero weight vector $0\neq w\in M$, we have $B\cdot w= M$ because $M$ is simple. The result then follows from part \eqref{dim1}.
\end{proof}

Let $M\in (B,R)\wmodO$, with $\cO=\ZZ\cdot \frm$ an infinite orbit. Then $M$ is a $\ZZ$-graded module for $B$, by defining $M_i=M_{\sigma^i(\frm)}$ for all $i\in\ZZ$. Clearly this grading is non-canonical, as we can choose any point in the orbit as our starting point. Choosing different starting points recovers the shifts of the modules in the graded module category.

Not all $\ZZ$-graded $B$-modules are $R$-weight modules. For example, $B$ itself is a $\ZZ$-graded $B$-module which is not an $R$-weight module. For a less trivial example, take $R=\CC[z_1,z_2]$, $\sigma(z_1)=z_1-1$, $\sigma(z_2)=z_2+1$, $H=(z_1-z_2)$, $J=(z_1-z_2-2)$ and let $\bp=H=(z_1-z_2)\in \Spec(R)$. Then $\bp\in\css(B)$ and $\sigma(\bp)=J\in\css(B)$, so we can define a graded module $(R/\bp)v_0\oplus (R/\sigma(\bp))v_1$ as in the proof of Proposition \ref{prop.props}\eqref{props2}. But notice that this is not a weight module because $\CC[h_1,h_2]/(h_1-h_2)$ does not decompose as a direct sum of weight spaces.

However, when we restrict to simple objects we do have a correspondence between weight modules and graded modules.

\begin{proposition}
Let $M$ be a simple $\ZZ$-graded $B$-module. Then $M$ is an $R$-weight module.
\end{proposition}
\begin{proof}
By assumption $M=\bigoplus_{i\in\ZZ}M_i$, with $B_jM_i\subseteq M_{i+j}$. Let $i\in\ZZ$ such that $M_i\neq 0$. Then $RM_i=B_0M_i\subseteq M_i$, so $M_i$ is an $R$-module. Suppose $M_i$ is not simple as an $R$-module, and let $0\neq N_i\subsetneq M_i$ be a nontrivial $R$-submodule.  Since $(BN_i)_i=N_i$, then $0\neq BN_i\subsetneq M$ so $BN_i$ is a nontrivial $B$-submodule of $M$, which is a contradiction. Hence $M_i$ is a simple $R$-module, so $M_i\simeq R/\frm$ as $R$-modules for some $\frm\in\Maxspec(R)$. 

Since $M$ is a simple module, $M=BM_i$. Hence, for all $j\in \ZZ$, we have $M_j=B_{j-i}M_i$. If $r\in \sigma^{j-i}(\frm)$, then
\[rM_j=rB_{j-i}M_i=B_{j-i}\sigma^{i-j}(r)M_i=0\]
because $\sigma^{i-j}(r)\in\frm$. It follows that $M_j\subseteq M_{\sigma^{j-i}(\frm)}$ and therefore
\[M=\bigoplus_{j\in \ZZ}M_j\subseteq \bigoplus_{j\in\ZZ}M_{\sigma^{j-i}(\frm)}\subseteq M,\]
which implies that the inclusions are actually equalities and indeed $M=\bigoplus_{j\in\ZZ}M_{\sigma^{j-i}(\frm)}$ is a weight module.
\end{proof}

\begin{lemma}\label{lem.actzero}
Let $\cO=\ZZ\cdot\frm \in\Maxspec(R)/\ZZ$ be an infinite orbit. Let $M$ be a simple weight module, and $0\neq w\in M_{\frm}$ a weight vector. If $j\in J\cap\sigma(\frm)$, then $jt\cdot w=0$. Analogously, if $h\in H\cap\frm$, then $\sigma\inv(h)t\inv\cdot w=0$. 
\end{lemma}

\begin{proof}
Let $j \in J \cap \sigma(\frm)$ and suppose that $w_1:=jt\cdot w\neq 0$. For $h \in H$ we have 
\[\sigma\inv(h)t\inv\cdot w_1=\sigma\inv(h)t\inv jt\cdot w=\sigma\inv(h)\sigma\inv(j)\cdot w=0.\]
The last equality follows because $j\in\sigma(\frm)$, so $\sigma\inv(j)\in\frm$. Since $\cO$ is an infinite orbit, $(Bw_1)_\frm=B_{-1}w_1=0$, so $w\not\in B_{-1}w_1$. Hence $w\not\in B w_1$ and so $0\neq B w_1\subsetneq M$ is a nontrivial proper weight submodule. This contradicts the simplicity of $M$. 
The second part is similar.
\end{proof}

Recall the definition of a \emph{break} given in Definition \ref{defn.breaks}. The next two lemmas define the weight modules that appear in our classification in the infinite orbit case.

\begin{lemma}
Let $\cO\in\Maxspec(R)/\ZZ$ be an infinite orbit with no breaks. The $R$-module
\[M(\cO)=\bigoplus_{\frm\in\cO}\left(R/\frm\right)v_{\frm}\]
is a simple $B$-weight module with $B$-action defined by
\[jt\cdot v_\frm= j v_{\sigma(\frm)},\qquad  \sigma\inv(h)t\inv \cdot v_\frm= \sigma\inv(h) v_{\sigma\inv(\frm)},\]
for all $j\in J$ and $h\in H$.
\end{lemma}
\begin{proof}
The fact that the action of $B$ on the module $M(\cO)$ 
is well defined can be verified exactly in the same way as in the proof of Proposition \ref{prop.props}\eqref{props2}. Furthermore, it is clear from the definition of the $B$-action that $M(\cO)$ is simple. 
\end{proof}

Let $\cO \in \Maxspec(R)/\ZZ$ be an infinite orbit and let $\emptyset\neq \beta\subseteq \cO$ be the set of breaks. We can define an order on $\cO$ by $\frm <\sigma(\frm)$ for all $\frm\in\cO$. Set 
\[
\beta' = \begin{cases}
\beta\cup\{\infty\} 
    & \text{if $\beta$ contains a maximal element} \\
\beta 
    & \text{otherwise.}
\end{cases}\]
For $\frn\in\beta'$, we let $\frn^-$ be the maximal element of $\beta'$ such that $\frn^-<\frn$, or $\frn^-=-\infty$ if $\frn$ was a minimal element of $\beta'$. We extend the order in $\cO$ to $\cO\cup\{\pm\infty\}$ in the obvious way.

\begin{lemma}
Let $\cO\in\Maxspec(R)/\ZZ$ be an infinite orbit, let $\emptyset\neq\beta \subseteq \cO$ be the set of breaks, and keep the notation above. For $\frn \in \beta'$, the $R$-module
\[M(\cO,\frn)=\bigoplus_{\frm\in\cO,~ \frn^-<\frm\leq\frn}\left(R/\frm\right)v_{\frm}\qquad (\text{if $\frn=\infty$, the second inequality is strict)}\]
is a simple $B$-weight module with $B$-action defined by 
\[
jt\cdot v_\frm=
\begin{cases}
j v_{\sigma(\frm)} & \text{ if }\frm<\frn \\ 
0 & \text{ if }\frm=\frn,
\end{cases}
\qquad 
\sigma\inv(h)t\inv\cdot v_\frm=
\begin{cases} \sigma\inv(h) v_{\sigma\inv(\frm)} & \text{ if }\sigma\inv(\frm)>\frn^- \\ 
0 & \text{ if }\sigma\inv(\frm)=\frn^-,
\end{cases}\]
for all $j\in J$ and $h\in H$.
\end{lemma}
\begin{proof}
Again, it is clear that $M(\cO,\frn)$ is simple and the action is well defined, as in Proposition \ref{prop.props}\eqref{props2}, because $\frn$ and $\frn^-$ are breaks.
\end{proof}

\begin{theorem}\label{thm.inf-orb-class}
Let $\cO\in\Maxspec(R)/\ZZ$ be an infinite orbit and let $\beta \subseteq \cO$ be the set of breaks.
\begin{enumerate}
    \item \label{inf1} If $\beta=\emptyset$, then up to isomorphism, $M(\cO)$ is the unique simple module in $(B,R)\wmodO$.
    \item \label{inf2} If $\beta \neq \emptyset$, then $\{M(\cO,\frn)\}_{\frn\in\beta'}$ is a complete list of isomorphism classes of simple modules in $(B,R)\wmodO$.
\end{enumerate}
\end{theorem}
\begin{proof}
\eqref{inf1}
If $M\in (B,R)\wmodO$ is a simple weight module, then by Proposition \ref{prop.dim-weight}\eqref{dim2} we have that for each $\frm\in\cO$,  either $M_\frm=0$ or $M_\frm\simeq R/\frm$ as $R$-modules. Let $M_\frm\neq 0$ and $0\neq w\in M_\frm$. Set $v_{\frm}:=w$. Since $\cO$ has no breaks, by Proposition \ref{prop.inj-maps}\eqref{inj1}, $B_1v_{\frm}\neq 0$ and $B_{-1}v_{\frm} \neq 0$. Hence, $B_{\pm 1}\cdot v_{\frm}\simeq  R/\sigma^{\pm 1}(\frm)$ as $R$-modules. Let $j_0\in J$ such that $j_0t\cdot v_{\frm}\neq 0$. By Lemma \ref{lem.actzero}, $j_0\not\in\sigma(\frm)$, so there is a $r_0\in R$ such that $r_0j_0=1+q$, with $q\in\sigma(\frm)$. Set $v_{\sigma(\frm)}:=(r_0j_0)t\cdot v_{\frm}$. Then, for all $j\in J\setminus \sigma(\frm)$ there is $r\in R$ such that $j-rr_0j_0\in\sigma(\frm)$. If follows that
\begin{align*}
jt\cdot v_{\frm}&=(j-rr_0j_0+rr_0j_0)t\cdot v_{\frm}
    = r v_{\sigma(\frm)}\\
    &= (r+rq-rq)v_{\sigma(\frm)} 
    = rr_0j_0v_{\sigma(\frm)}\\
    &= (rr_0j_0-j+j)v_{\sigma(\frm)}
    = jv_{\sigma(\frm)}. 
\end{align*}
Notice that for all $j\in J\cap\sigma(\frm)$ we have $jt\cdot v_{\frm}=0$ by Lemma \ref{lem.actzero}. Analogously, we can define $v_{\sigma\inv(\frm)}$ by looking at the action of $\sigma\inv(h)t\inv$ on $v_{\frm}$ for any $h\in H\setminus\sigma(\frm)$. Proceeding inductively, we obtain that $M\simeq M(\cO)$.
 
\eqref{inf2} 
If $M$ is a simple weight module with $M_{\frm}\neq 0$, it is clear from Lemma \ref{lem.actzero} that the support of $M$ is contained in the interval (according to the order on $\cO$) between two breaks, which contains $\frm$. The fact that $M \simeq M(\cO,\frn)$, where $\frn\in\beta'$ is the minimal element with $\frm\leq \frn$, follows from the same reasoning as in part \eqref{inf1} (there are no further breaks within the interval, so we can define $v_\frm$ and $v_{\sigma(\frm)}$ as desired).
\end{proof}

Indecomposable weight modules for GWAs have been described in \cite{DGO} and for infinite orbits their classification can be easily obtained from the classification of the simples. While the description of simple weight modules for BR algebras is similar to the situation for GWAs, the indecomposables can be much more complicated.

\begin{example}
Let $B$ be the algebra of Example \ref{ex.nongwa}. Recall that $R=\kk[z_1,z_2]$, $\sigma(z_1)=z_1-1$, $\sigma(z_2)=z_2+1$, $H=\sigma(z_1,z_2)$, $J=R$. For this example we also assume that $\operatorname{char}(\kk)=0$. Then $HJ=\sigma(z_1,z_2)=(z_1-1,z_2+1)$ is a maximal ideal. We consider the infinite orbit 
\[ \cO=\ZZ\cdot \frm=\{(z_1-k,z_2+k)\in\Maxspec(R)~|~k\in\ZZ\},\]
which contains the only break $\frm=(z_1,z_2)$. All the other orbits have no breaks, hence the categories of weight modules supported on those other orbits are semisimple with a single simple object up to isomorphism.

(1) Consider $$M=\bigoplus_{k\in \ZZ}(R/(z_1-k,z_2+k))v_k$$ as an $R$-module. We define a $B$-action by
\[t\cdot v_k=v_{k+1}; \qquad 
z_it\inv\cdot v_k=\begin{cases} 
    z_i v_{k-1} & \text{ if }k\neq 1 \\ 
    0 & \text{ if }k=1. 
\end{cases}\]
Then $M$ is indecomposable but not simple as the module spanned by $\{v_k~|~k\geq 1\}$ is a proper submodule. This is very similar to the GWA situation.

(2) Let $\alpha_1,\alpha_2\in \kk$, \[M(\alpha_1,\alpha_2)=\bigoplus_{k\in \ZZ}(R/(z_1-k,z_2+k))v_k\] 
as an $R$-module. We define a $B$-action by
\[ t\cdot v_k=\begin{cases} 
    v_{k+1} & \text{ if }k\neq 0 \\ 
    0 & \text{ if }k=0, 
\end{cases} \qquad 
z_it\inv\cdot v_k=\begin{cases} 
    z_i v_{k-1} & \text{ if }k\neq 1 \\ 
    \alpha_i v_0 & \text{ if }k=1. 
\end{cases}\]
As long as $(\alpha_1,\alpha_2)\neq (0,0)$, $M(\alpha_1,\alpha_2)$ is indecomposable. Since the module spanned by $\{v_k~|~k\leq 0\}$ is a proper submodule, then $M(\alpha_1,\alpha_2)$ is not simple. We also have $M(\alpha_1,\alpha_2)\simeq M(\alpha'_1,\alpha'_2)$ if and only if $(\alpha_1,\alpha_2)=c(\alpha'_1,\alpha'_2)$ for some $c\in \kk$. Thus, there is a $\PP_{\kk}^1$-family of nonisomorphic, nonsimple indecomposable modules that have the same composition series of length two. This is different from the situation for GWAs, where there would be only one isomorphism class of such modules.

(3) Let 
\[ N=\bigoplus_{k\leq 0}\Big(R/(z_1-k,z_2+k)\Big)v_k\oplus\bigoplus_{k\geq 1}\Bigg(\Big(R/(z_1-k,z_2+k)\Big)v_k+\Big(R/(z_1-k,z_2+k)\Big)w_k\Bigg)\] 
as an $R$-module. We define a $B$-action by
\begin{align*}
t\cdot v_k &=
    \begin{cases} 
        v_{k+1} & \text{ if }k\neq 0 \\ 
        0 & \text{ if }k=0, 
    \end{cases} & 
z_it\inv\cdot v_k &= 
    \begin{cases} 
        z_i v_{k-1} & \text{ if }k\neq 1 \\ 
        (i-1) v_0 & \text{ if }k=1, 
    \end{cases} \\
t\cdot w_k &= w_{k+1}\quad k\geq 1, & 
    z_it\inv\cdot w_k &=
        \begin{cases}
            z_i w_{k-1} & \text{ if }k> 1 \\ 
            (2-i) v_0 & \text{ if }k=1. 
        \end{cases}
\end{align*}
Then $N$ is an indecomposable module where the weight spaces are not all one-dimensional, which does not happen for GWAs.
\end{example}

\subsection{Finite orbit case}\label{sec.finorb}

Now we want to examine the situation for finite orbits. The following hypothesis will be assumed throughout the rest of this section.

\begin{hypothesis}\label{hyp.finorb}
We fix $n\in \NN$ and we suppose that $\sigma^n=\id_R$. We further suppose that $\cO = \ZZ \cdot \frm \subset \Maxspec(R)$ is an orbit of size $n$.
\end{hypothesis}

This hypothesis is not necessary to have a finite orbit. If $\frm \in \Maxspec(R)$ satisfies $\sigma^k(\frm) \subseteq \frm$, then $\ZZ \cdot \frm$ will be a finite orbit. Further, even if $\sigma^n = \id_R$, it is possible that some finite orbits have sizes which are proper divisors of $n$.

However, these two simplifying assumptions hold in most relevant examples and restricting our attention to orbits satisfying Hypothesis~\ref{hyp.finorb} allows us to describe the simple weight modules.

Under these hypotheses, we have $\cO=\{\frm,\sigma(\frm),\ldots,\sigma^{n-1}(\frm)\}$. We let $\Lambda_n:=\oplus_{k\in\ZZ}B_{kn}\subset B$ be the $n$-th Veronese subalgebra of $B$. For any $\frm\in\cO$, we have $B_n\frm=\sigma^n(\frm)B_n=\frm B_n$, and $B_{-n}\frm=\sigma^{-n}(\frm)B_{-n}=\frm B_{-n}$, hence $\frm \Lambda_n=\Lambda_n \frm$ is a two sided ideal in $\Lambda_n$.

Notice that $\sigma(I^{(n)})=\sigma(J\sigma(J)\cdots \sigma^{n-1}(J))=\sigma(J)\sigma^2(J)\cdots \sigma^n(J)=\sigma(J)\sigma^2(J)\cdots J=I^{(n)}$, and similarly $\sigma(I^{(-n)})=I^{(-n)}$, hence $\sigma$ restricts to an automorphism $\sigma|_{\Lambda_n}:\Lambda_n\sto \Lambda_n$ where $\sigma(a t^{kn})=\sigma(a)t^{kn}$. The following is then clear.

\begin{lemma}\label{lem.iso-lambda}
For all $\frm\in\cO$ and for all $i\in\ZZ$, there is an isomorphism of algebras $\Lambda_n/(\frm)\simeq \Lambda_n/(\sigma^i(\frm))$ induced by $\sigma^i$.
\end{lemma}

\begin{notation}
Let $N$ be a $\Lambda_n/(\frm)$-module, then for all $i\in\ZZ$, we define the $\Lambda_n/(\sigma^i(\frm))$-module ${}^{\sigma^i} N$. As a $\kk$-vector space, ${}^{\sigma^i} N = N$. For $v \in N$, we use the notation ${}^{\sigma^i} v$ for the corresponding element in ${}^{\sigma^i} N$. The action is defined by
\[ \lambda\cdot {}^{\sigma^i}w:= \sigma^{-i}(\lambda) {}^{\sigma^i}w \qquad \text{ for all } {}^{\sigma^i}w\in {}^{\sigma^i}N \text{ and all }\lambda\in\Lambda_n/(\sigma^i(\frm)).\]
Similarly, if $\theta: N \sto N$ is an $R/\frm$-linear map, then we use the notation ${}^{\sigma^i} \theta$ for the induced $R/\sigma^i(\frm)$-linear map ${}^{\sigma^i}N \sto {}^{\sigma^i}N$.
\end{notation}

The next set of results give a sense of the structure of weight modules in the case in which $\cO=\ZZ\cdot\frm$ has no breaks. That is, $HJ\not\subset\frm$ for all $\frm\in\cO$.

\begin{lemma}\label{lem.simpleLam-mod}
Suppose that there are no breaks in $\cO$ and that $M=\bigoplus_{\frm\in\cO}M_\frm$ is a simple weight $B$-module. Then for each $\frm \in \cO$, $M_\frm$ is a simple $\Lambda_n/(\frm)$-module. Further, for all $\frm \in \cO$ and all $i\in \ZZ$, $M_{\sigma^{i}(\frm)}\simeq {}^{\sigma^i} M_\frm$.
\end{lemma}
\begin{proof}
Let $\frm\in\cO$. If $b\in B$, then $b\cdot M_\frm\subset M_\frm$ if and only if $b\in\Lambda_n$. Also, $\frm\cdot M_\frm=0$, so $M_\frm$ is indeed a $\Lambda_n/(\frm)$-module. Now suppose that $0\subsetneq N\subsetneq M_\frm$ is a nontrivial $\Lambda_n/(\frm)$-submodule, then $B \cdot N$ is a $B$-submodule of $M$ and $(B\cdot N)_\frm=N$. Hence, $B\cdot N$ is a proper submodule, which contradicts the simplicity of $M$. Finally, since there are no breaks in the orbit, then Proposition \ref{prop.inj-maps} \eqref{inj1} implies that $jt:M_\frm\sto M_{\sigma(\frm)}$ is a bijection for any $j\in J\setminus\sigma(\frm)$. Pick any such $j$, let $\lambda t^{kn}\in B_{kn}$ for some $k\in\ZZ$, and let $w\in M_\frm$. Then
\[\lambda t^{kn}(jt\cdot w)=\lambda \sigma^{kn}(j)t^{kn+1}\cdot w=\lambda j t^{1+kn}\cdot w = jt(\sigma^{-1}(\lambda)t^{kn}\cdot w) \]
which shows exactly that $M_{\sigma(\frm)}\simeq {}^\sigma M_\frm$ and, by induction, that $M_{\sigma^{i}(\frm)}\simeq {}^{\sigma^i} M_\frm$.
\end{proof}

\begin{lemma}\label{lem.Lambdathetamod}
Suppose that there are no breaks in $\cO$, let $\frm \in \cO$, and let $N$ be a $\Lambda_n/(\frm)$-module. Then there exists an invertible $R/\frm$-linear transformation $\theta \in \Aut_{R/\frm}(N)$ such that, as elements of $\End_{R/\frm}(N)$, we have
\begin{equation}\label{eq.thetaaction} 
a t^n := 
\begin{cases}
a\theta  & \text{ if }a\in I^{(n)}\setminus\frm \\ 
0 & \text{ if }a\in I^{(n)}\cap\frm
\end{cases}
\qquad 
\text{and}
\qquad
b t^{-n} :=
\begin{cases}
b\theta \inv & \text{ if }b\in I^{(-n)}\setminus\frm \\ 
0 & \text{ if }b\in I^{(-n)}\cap\frm
\end{cases}
\end{equation}
for all $a \in I^{(n)}$ and $b \in I^{(-n)}$.

Conversely, if $N$ is a $R/\frm$-vector space and $\theta \in \Aut_{R/\frm}(N)$ is an invertible linear transformation, then $N$ becomes a $\Lambda_n/(\frm)$-module by defining the action of the generators as in \eqref{eq.thetaaction}. Finally, $N$ is a simple $\Lambda_n/(\frm)$-module if and only if $N$ has no nontrivial $\theta $-invariant subspaces.
\end{lemma}
\begin{proof}
Since $HJ\not\subset \sigma^i(\frm)$ for all $i$, we have that $\sigma^i(H),\sigma^i(J)\not\subset\frm$, for all $i$, which then implies that $I^{(n)}, I^{(-n)}\not\subset\frm$. Let $N$ be a $\Lambda_n/(\frm)$-module and let $a\in I^{(n)}\setminus \frm$, $b\in I^{(-n)}\setminus\frm$. Then $(at^n)(bt^{-n})=ab\in (R/\frm) \id_N$ and since $ab \not \in \frm$, this map is invertible. Define $\theta :=\check{a}(at^n)$, where $\check{a}a\equiv 1\pmod \frm$. Then $\theta $ is an $R/\frm$-invertible linear transformation that satisfies all the required conditions. Note also that a proof similar to the proof of Theorem~\ref{thm.inf-orb-class}  shows that $\theta $ is independent of the choice of $a$. If $a\in I^{(n)}\cap \frm$ and $b\in I^{(-n)}\setminus\frm$, then $(at^n)(bt^{-n})=ab=0\in (R/\frm) \id_N$. Hence $at^n=0$ since $bt^{-n}$ is invertible. The same reasoning applies if $a\in I^{(n)}\setminus \frm$ and $b\in I^{(-n)}\cap\frm$. 

Given a $R/\frm$-module $N$, it is immediate to verify that \eqref{eq.thetaaction} gives a well-defined $\Lambda_n/(\frm)$-action. Since $B_{\pm n} N=(R/\frm) \theta ^{\pm 1}N$ by the above, it is clear that $N$ is a simple $\Lambda_n/(\frm)$-module if and only if there are no nontrivial $\theta $-invariant subspaces.
\end{proof}

\begin{lemma}\label{lem.simples-nobr}
Suppose that $\cO$ does not contain any breaks. Let $\frm\in\cO$, let $N$ be a $R/\frm$-vector space, and let $\theta \in \Aut_{R/\frm}(N)$ be an invertible linear transformation. We can consider $N$ as a $\Lambda_n/(\frm)$-module by Lemma  \ref{lem.Lambdathetamod}. We define 
\[M(\frm,N,\theta )=\bigoplus_{i=0}^{n-1} {}^{\sigma^i}N.\]
Then $M(\frm,N,\theta )$ is a weight module for $B$ with $B$-action defined as follows. For all $j\in J$, $h\in H$, and $v \in N$, we define
\begin{align}\label{eq.jact}
jt({}^{\sigma^i}v) &=
\begin{cases}
j({}^{\sigma^{i+1}}v) & \text{ if }0\leq i\leq n-2, \\
j \theta v & \text{ if }i=n-1,
\end{cases} \\\
\label{eq.hact}
\sigma^{-1}(h)t^{-1}({}^{\sigma^i}v) &=
\begin{cases} \sigma^{-1}(h)({}^{\sigma^{i-1}}v) & \text{ if }1\leq i\leq n-1, \\
\sigma^{-1}(h) \theta ^{-1}v & \text{ if }i=0.
\end{cases} 
\end{align}

If $N$ is simple as a $\Lambda_n/(\frm)$-module, then $M(\frm,N,\theta )$ is a simple weight module for $B$. Further,  for all $k\in \ZZ$ we have $M(\sigma^k(\frm),{}^{\sigma^k}N,{}^{\sigma^k}\theta )\simeq M(\frm,N,\theta ) $ as $B$-modules, where ${}^{\sigma^k}\theta $ is the $R/\sigma^k(\frm)$-linear transformation of ${}^{\sigma^k}N$ induced by $\theta $.
\end{lemma}
\begin{proof}
First we need to show that the action is well defined. Let $a\in I^{(n)}$. Then $a=j_1\sigma(j_2)\cdots \sigma^{n-1}(j_n)$, with $j_i\in J$ and for any ${}^{\sigma^i}v\in {}^{\sigma^i}(N)$,
\begin{align*}
a({}^{\sigma^i}\theta  ({}^{\sigma^i}v))
    &=(at^n)\cdot ({}^{\sigma^i}v) \\
    &=(j_1t)(j_2t)\cdots (j_nt)\cdot ({}^{\sigma^i}v) \\
    &=(j_1t)\cdots (j_{i}t)(j_{i+1}t)\cdot j_{i+2}\sigma(j_{i+3})\ldots \sigma^{n-2-i}(j_n)({}^{\sigma^{n-1}}v) \\
    &=(j_1t)\cdots (j_{i}t)\cdot j_{i+1}\sigma(j_{i+2})\sigma^2(j_{i+3})\ldots \sigma^{n-1-i}(j_n)\theta v \\
    &= j_1\sigma(j_2)\cdots \sigma^{n-1}(j_n)^{\sigma^i}\theta  ({}^{\sigma^i}v) \\
    &= a {}^{\sigma^i}\theta  ({}^{\sigma^i}v).
\end{align*}
A similar computation shows that the action of $I^{(-n)}t^{-n}$ is also well defined.

Suppose $0\neq w\in M(\frm,N,\theta )$ is a weight vector, then $w\in {}^{\sigma^i}N$ for some $i=0,\ldots, n-1$ since those are the weight spaces. If $N$ is a simple $\Lambda_n/(\frm)$-module, then ${}^{\sigma^i}N$ is a simple $\Lambda_n/\sigma^i(\frm)$-module, hence $\Lambda_n\cdot w={}^{\sigma^i}N$. Since there are no breaks in the orbit, there exists $j\in J\setminus\sigma^i(\frm)$ and by Proposition \ref{prop.inj-maps} \eqref{inj1}, $jt({}^{\sigma^i}N)={}^{\sigma^{i+1}}N$ for all $i$. It follows that $Bw=M$ for any nonzero weight vector, hence $M$ is a simple weight module.

Now, it is clear that for any $k\in \ZZ$, $M(\sigma^k(\frm),{}^{\sigma^k}N,{}^{\sigma^k}\theta )\simeq M(\frm,N,\theta )$ as $R$-modules, since for each of them, the weight space with weight $\sigma^i(\frm)$ is isomorphic to ${}^{\sigma^i}N$ for all $k$. More precisely, we define an isomorphism $\alpha:M(\frm,N,\theta )\to M(\sigma^k(\frm),{}^{\sigma^k}N,{}^{\sigma^k}\theta )$ as follows:
\[\alpha(w)=\begin{cases}
w & \text{ if }w\in {}^{\sigma^i}N,~0\leq i\leq k-1  \\ 
{}^{\sigma^i}\theta w & \text{ if }w\in {}^{\sigma^i}N,~k\leq i\leq n-1 .\end{cases}\]
Then it is clear from the definition that
\[\alpha(jt\cdot w)=jt\cdot\alpha(w),\qquad \alpha(\sigma\inv(h)t\inv\cdot w)=\sigma\inv(h)t\inv\cdot\alpha(w)\]
for all $j\in J$, $h\in H$, and $w\in M(\frm, N,\theta )$.
\end{proof}

\begin{definition}\label{def.HJbreaks}
Suppose that $\frm\in\cO$ is a break and let $M\in (B,R)\wmodO$ be a weight module. We say that $\frm$ is a \emph{$J$-break} for $M$ if there exists $j\in J$, $v\in M_\frm$ such that $jt\cdot v\neq0$. Similarly, we say that $\frm$ is an \emph{$H$-break} for $M$ if there exists $h\in H$ and $w\in M_{\sigma(\frm)}$ with $\sigma^{-1}(h)t^{-1}\cdot w\neq 0$.
\end{definition}

\begin{lemma}\label{lem.HJbreaks}
If $\cO$ is a finite orbit with breaks, $M\in (B,R)\wmodO$ is a simple weight module, and $\frm$ is a break, then $\frm$ cannot be both a $J$-break and an $H$-break for $M$. Furthermore, if $\frm$ is a $J$-break (resp. $H$-break) for $M$, then all other breaks in the orbit $\cO$ are also $J$-breaks (resp. $H$-breaks) for $M$. In this case, $\Supp(M)=\cO$.
\end{lemma}
\begin{proof}
Suppose that $\frm$ is a $J$-break. The case in which $\frm$ is an $H$-break is completely analogous. Then there exists $j\in J$ and $v\in M_\frm$ such that $jt\cdot v\neq0$. 
Let $0\neq w=jtv\in M_{\sigma(\frm)}$, then for any $h\in H$ we have that
\[\sigma^{-1}(h)t^{-1}\cdot w=\sigma^{-1}(h)t^{-1}(jtv)=\sigma^{-1}(hj)v=0\]
because $\sigma^{-1}(hj)\in\frm$. This implies that $B_{-k}w=0$ for all $k>0$, and that, by the simplicity of $M$, $M=Bw=B_{\geq 0}w$. It follows that for all $u\in M_{\sigma(\frm)}$ there exist $j_1,\ldots,j_{rn}\in J$ such that $u=(j_1t)\cdots (j_{rn}t)w$. So, for all $h\in H$, 
\begin{align*} 
\sigma^{-1}(h)t^{-1}u&=\sigma^{-1}(h)t^{-1}(j_1t)\cdots (j_{rn}t)w \\
    &= \sigma^{-1}(hj_1)(j_2t)\cdots (j_{rn}t)w \\
    &=(j_2t)\cdots (j_{rn}t)\sigma^{-(rn-1)-1}(hj_1)w \\
    &=(j_2t)\cdots (j_{rn}t)(hj_1)w
    =0, 
\end{align*}
which proves that $\frm$ is not an $H$-break. The same reasoning proves that all the other breaks cannot be $H$-breaks. 

We now show that the other breaks must be $J$-breaks. Since $M=B_{\geq 0}w$, let $0\neq v\in M_\frm$ as above. Then there exist $j'_1,\ldots,j'_{sn-1}\in J$ such that $0\neq v=(j'_1t)\cdots (j'_{sn-1}t)w$. For all $i$, $0\neq w_i=(j'_it)\cdots(j'_{sn-1}t)w$ and the various $w_i$ belong to all the other weight spaces in the orbit. Hence, $\Supp(M)=\cO$. In particular, if $w_i$ corresponds to a break, then $0\neq w_{i-1}=(j_{i-1}t)w_i$. Thus, all the breaks are indeed $J$-breaks.
\end{proof}

Our next theorem classifies simple weight modules under Hypothesis \ref{hyp.finorb}. We note that the modules $M(\frm,N,\theta )$ appearing in the theorem are those defined by Lemma \ref{lem.simples-nobr}. In the next result we classify isomorphisms between these objects.

\begin{lemma}
Suppose that $\cO=\ZZ\cdot \frm$ is a finite orbit. Let $N$ (resp. $N'$) be a vector space over $R/\frm$ and let $\theta \in \Aut_{R/\frm}(N)$ (resp. $\theta' \in \Aut_{R/\frm}(N')$) be an invertible linear transformation that leaves no nontrivial subspace of $N$ (resp. $N'$) invariant. Then $M(\frm,N,\theta )\simeq M(\frm,N',\theta ')$ if and only if there is an isomorphism of $R/\frm$-vector spaces $\varphi:N\to N'$ with $\varphi \theta =\theta '\varphi$.
\end{lemma}
\begin{proof}
Suppose that $\check{\varphi}:M(\frm,N,\theta )\to M(\frm,N',\theta ')$ is an isomorphism of simple weight modules for $B$. Then $\varphi:=\check{\varphi}|_N:N\to N'$ is an isomorphism of $R/\frm$-vector spaces. Let $a\in I^{(n)}\setminus \frm$ and $a'\in R$ with $a'a\equiv 1\pmod \frm$. Then for all $v\in N$, we have
\[\varphi\theta (v)=\varphi(a'a\theta v)=\varphi(a'at^nv)=a'at^n\varphi(v)=a'a\theta '\varphi(v)=\theta '\varphi(v)\]
so indeed $\varphi \theta =\theta '\varphi$. Conversely, if $\varphi:N\to N'$ is an isomorphism and $\theta ,\theta '$ are invertible linear transformations such that $\varphi \theta =\theta '\varphi$, then $\varphi$ is an isomorphism of $\Lambda_n/\frm$-modules. It follows that the twists ${}^{\sigma^i}\varphi:{}^{\sigma^i}N\to{}^{\sigma^i}N'$ give isomorphisms of $\Lambda_n/\sigma^i(\frm)$-modules and by defining $\check{\varphi}:=\oplus_{i=0}^{n-1}{}^{\sigma^i}\varphi$ we obtain the required isomorphism $\check{\varphi}:M(\frm,N,\theta )\to M(\frm,N',\theta ')$.
\end{proof}

Suppose that $\cO=\ZZ\cdot \frm$ is a finite orbit with breaks. Let $\{ \sigma^i(\frm) \mid i \in I\}$ with $I \subset \{0,\hdots,n-1\}$, $|I|=m$, be the set of breaks. For each $i_k$ with $1 \leq k < m$, we define the $R$-module
\[ M(\cO,i_k)=\bigoplus_{\ell=i_k+1}^{i_{k+1}}R/\sigma^{\ell}(\frm)v_\ell.\]
By a similar argument as in Proposition  \ref{prop.props}\eqref{props2}, there is a well-defined $B$-action on $M(\cO,i_k)$ given by
\[jt\cdot v_\ell=
\begin{cases}
jv_{\ell+1} & \text{ if }\ell<i_{k+1} \\ 
0 & \text{ if }\ell=i_{k+1},
\end{cases}
\qquad
\sigma\inv(h)t\inv\cdot v_\ell=
\begin{cases}
\sigma\inv(h)v_{\ell-1} & \text{ if }\ell>i_k+1 \\ 
0 & \text{ if }\ell=i_{k}+1,\end{cases}\]
for all $j \in J$ and $h \in H$. Similarly, we define
\[ M(\cO,i_m)=\bigoplus_{\ell=i_m+1}^{n+i_{1}}R/\sigma^{\ell}(\frm)v_\ell\]
with $B$-action given by
\[
jt\cdot v_\ell =
\begin{cases}jv_{\ell+1} & \text{ if }\ell<n+i_{1} \\ 
0 & \text{ if }\ell=n+i_{1},\end{cases} \qquad  \sigma\inv(h)t\inv\cdot v_\ell=
\begin{cases}
\sigma\inv(h)v_{\ell-1} & \text{ if }\ell>i_m+1 \\ 
0 & \text{ if }\ell=i_{m}+1,\end{cases}
\]
for all $j\in J$ and $h\in H$.

We can now state our classification of the simple weight modules supported on a finite orbit. When there are no breaks on the orbit, the structure is analogous to the situation for rank 1 GWAs, with the module depending on the choice of a linear transformation with no invariant subspaces, which in \cite{DGO} was taken to be the companion matrix of an irreducible polynomial. When there are breaks, some simples are supported on the interval between two breaks, similarly to the modules for GWAs. However, there are also modules with full support and, unlike in \cite{DGO}, we cannot  describe them explicitly.

\begin{theorem}\label{thm.fin-orb-class}
Suppose that $\cO=\ZZ\cdot \frm$ is a finite orbit.
\begin{enumerate}
\item \label{fin1} Assume there are no breaks for $B$ on $\cO$ and fix $\frm\in\cO$. Up to isomorphism, the simple weight modules in $(B,R)\wmodO$ are of the form $M(\frm, N ,\theta )$ where where $N$ is a vector space over $R/\frm$ and $\theta \in \Aut_{R/\frm}(N)$ is an invertible linear transformation that leaves no nontrivial subspace invariant.

\item \label{fin2} Assume there are breaks. Let $\{\sigma^{i_1}(\frm),\ldots,\sigma^{i_m}(\frm)~|~0\leq i_1<\cdots <i_m\leq n-1\}\neq \emptyset$ be the set of breaks. Then all simple weight modules in $(B,R)\wmodO$ are of one of the following three kinds.

\begin{enumerate}
\item $M(\cO,i_k)$, $k=1,\hdots,m$, as defined above.
\item $M$ such that $\Supp(M)=\cO$ and all breaks are $J$-breaks as in Definition \ref{def.HJbreaks}.
\item $M$ such that $\Supp(M)=\cO$ and all breaks are $H$-breaks as in Definition \ref{def.HJbreaks}.
\end{enumerate}
\end{enumerate}
\end{theorem}
\begin{proof}
\eqref{fin1}
Suppose that $B$ has no breaks on $\cO$ and that $M\in(B,R)\wmodO$ is a simple weight module. Let $\frm\in\cO$, then $N=M_\frm$ is a simple $\Lambda_n/(\frm)$-module by Lemma~\ref{lem.simpleLam-mod}, which, by Lemma~\ref{lem.Lambdathetamod}, implies that there exists $\theta \in \Aut_{R/\frm}(N)$ which satisfies \eqref{eq.thetaaction} and that leaves no nontrivial subspaces invariant. We claim that $M\simeq M(\frm,N,\theta )$ as weight $B$-modules. By Lemma~\ref{lem.simpleLam-mod}, $M_{\sigma^i(\frm)}\simeq {}^{\sigma^i}N$ as $R/\sigma^i(\frm)$-modules for all $i=0, \ldots, n-1$. Let $v\in N$, since $J\not\subset\sigma^i(\frm)$, for all $i$, pick $j_0,\ldots,j_{n-1}\in J$ such that $j_i\not\in \sigma^i(\frm)$. For $i=1,\ldots,n-1$, define inductively $v={}^{\sigma^0}v$, ${}^{\sigma^{i}}v:=r_{i-1}(j_{i-1}t) ({}^{\sigma^{i-1}}v) \in M_{\sigma^i(\frm)}$ where $r_{i-1}\in R\setminus\sigma^i(\frm)$ is such that $r_{i-1}j_{i-1}\equiv 1\pmod {\sigma^i(\frm)}$. Notice that by construction  $(jt)({}^{\sigma^{i}}v)=j({}^{\sigma^{i+1}}v)$ hence \eqref{eq.jact} is satisfied for $0\leq i\leq n-2$. Finally, we have that for all $j\in J$
\begin{align*}
jt({}^{\sigma^{n-1}}v)
    &=jt(r_{n-1}j_{n-1}t)\cdots (r_1j_1)v \\
    &= j\sigma(r_{n-1}j_{n-1})\cdots\sigma^{n-1}(r_1j_1)t^n v \\
    &=j\sigma(r_{n-1}j_{n-1})\cdots\sigma^{n-1}(r_1j_1) \theta v \quad\text{by \eqref{eq.thetaaction}} \\
    &= j\theta v,
\end{align*}
where the last equality follows since $\sigma^{n-i}(r_{i}j_{i})\equiv 1\pmod \frm$ for all $i$.

The computation to verify \eqref{eq.hact} is entirely analogous and will be omitted.


\eqref{fin2} Now suppose that there are breaks for $B$ in $\cO$. As in the proof of Proposition  \ref{prop.props}\eqref{props2}, the modules $M(\cO,i_k)$ are simple. Let $M$ be a simple weight module with support in $\cO$. From Lemma \ref{lem.HJbreaks}, if any of the breaks in $\cO$ is a $J$-break (resp. $H$-break) for $M$, then $\Supp(M)=\cO$ and all breaks satisfy the same condition, so we are in one of the cases (b) or (c). Otherwise, if there are no $J$-breaks nor $H$-breaks for $M$ on $\cO$ and $M_{\sigma^r(\frm)}\neq 0$ for some $r$, then the support of $M$ has to be contained between the two consecutive breaks such that $i_k<r\leq i_{k+1}$ (or $i_m$ and $i_1$ if $r>i_m$). It then follows from the same argument as in the proof of Theorem \ref{thm.inf-orb-class}(2) that $M$ is isomorphic to $M(\cO,i_k)$.
\end{proof}

We conclude with an example that suggests giving a complete classification of all the simple weight modules with $J$-breaks and $H$-breaks is a hard problem.

\begin{example}\label{ex.finorb-mod}
Let $R=\CC[z_1,z_2]$, $q=e^{\frac{2\pi}{3}i}$, so that $q^3=1$. We set $\sigma(z_1)=qz_1$, $\sigma(z_2)=q\inv z_2=q^2z_2$, $J=(z_1-1,z_2-1)$, $H=\sigma(J)=(z_1-q^2,z_2-q)$. The algebra $B$ is then generated by $R$ together with $x_1=(z_1-1)t$, $x_2=(z_2-1)t\in Jt$, $y_1=(z_1-1)t\inv$, $y_2=(z_2-1)t\inv\in \sigma\inv(H)t\inv$.

Consider the orbit $\cO=\ZZ\cdot \frm$ with $\frm=J=(z_1-1,z_2-1)$, Then 
\[ \cO=\{\frm,\sigma(\frm),\sigma^2(\frm)\}= \{(z_1-1,z_2-1),(z_1-q^2,z_2-q),(z_1-q,z_2-q^2)\}.\]
Since $HJ=\frm\sigma(\frm)$, the breaks are $\{\frm,\sigma^2(\frm)\}$, so $i_1=0$ and $i_2=2$. We have two simple weight modules of the kind described in Theorem \ref{thm.fin-orb-class}(2a).
\begin{itemize}
\item The first is
\[ M(0)=\CC[z_1,z_2]/(z_1-q^2,z_2-q)v_1\oplus \CC[z_1,z_2]/(z_1-q,z_2-q^2)v_2\simeq \CC v_1\oplus \CC v_2\] 
with $B$-action given by
\begin{align*}
x_1v_1 &= (q-1)v_2,  & x_2v_1 &= (q^2-1)v_2, &
x_1v_2 &= x_2v_2 = 0, \\
y_1v_2 &= (q^2-1)v_1, & y_2v_2 &= (q-1)v_1, &
y_1v_1 &= y_2v_1=0.
\end{align*}
\item The second is \[M(2)=\CC[z_1,z_2]/(z_1-1,z_2-1)v_3\simeq \CC v_3\] 
with $B$-action given by $x_1v_3=x_2v_3=y_1v_3=y_2v_3=0$.
\end{itemize}
Additionally, we have the simple weight modules supported on the whole orbit. If we assume that all the weight spaces are one-dimensional over $\CC[z_1,z_2]/\sigma^i(\frm)\simeq \CC$, then 
these are described by cases (2b) and (2c) of Theorem \ref{thm.fin-orb-class}. First, set
\begin{align*}
M &= \CC[z_1,z_2]/(z_1-1,z_2-1)v_0\oplus \CC[z_1,z_2]/(z_1-q^2,z_2-q)v_1\oplus \CC[z_1,z_2]/(z_1-q,z_2-q^2)v_2 \\
&\simeq \CC v_0\oplus \CC v_1\oplus \CC v_2.
\end{align*}
Since $\sigma(\frm)$ is not a break, we have
\[ x_1v_1=(q-1)v_2, \quad x_2v_1=(q^2-1)v_2, \quad y_1v_2=(q^2-1)v_1,\quad y_2v_2=(q-1)v_1.\]
The remaining actions depend on whether the module has $J$-breaks or $H$-breaks.
\begin{itemize}
\item If we have $J$-breaks, then
\begin{align*}
y_1v_1&=y_2v_1=y_1v_0=y_2v_0=0, \\
x_1v_0&=(q^2-1)v_1,\quad x_2v_0=(q-1)v_1. 
\end{align*}
Now for any $(0,0)\neq (\alpha,\beta)\in\CC^2$ we have $x_1v_2=\alpha v_0$ and $x_2v_2=\beta v_0$. Note if $(\alpha,\beta)=(0,0)$ then $\sigma^2(\frm)$ would not be a $J$-break. 
    
\item If we have $H$-breaks, then
\begin{align*}
x_1v_2 &=x_2v_2=x_1v_0=x_2v_0=0, \\
y_1v_0&=(q-1)v_2,\quad y_2v_0=(q^2-1)v_2.
\end{align*}
Now for any $(0,0)\neq (\alpha,\beta)\in\CC^2$ we have $y_1v_1=\alpha v_0$ and $y_2v_1=\beta v_0$. Note if $(\alpha,\beta)=(0,0)$ then $\frm$ would not be a $H$-break. 
\end{itemize}

However, unlike the case of GWAs over $\CC$, we can have simple weight modules of the kinds from Theorem \ref{thm.fin-orb-class}(2b)-(2c) where the weight spaces are not all one-dimensional, as follows. Set
\begin{align*}
M&=\bigoplus_{i=1}^d\left(\CC[z_1,z_2]/(z_1-1,z_2-1)v_{0i}\oplus \CC[z_1,z_2]/(z_1-q^2,z_2-q)v_{1i}\oplus \CC[z_1,z_2]/(z_1-q,z_2-q^2)v_{2i}\right) \\
&\simeq\bigoplus_{i=1}^d\left( \CC v_{0i}\oplus \CC v_{1i}\oplus \CC v_{2i}\right).
\end{align*}
For any $A,A'\in M_d(\CC)$, considered as linear transformations of the weight space $V:=\bigoplus_{i=1}^d \CC v_{0i}$, we define a $B$-action as follows:
\begin{align*}
x_1v_{1i} &= (q-1)v_{2i}, & 
    x_1v_{0i} &= (q^2-1)v_{1i}, &
    x_1v_{2i} &= A v_{0i}, \\
x_2v_{1i} &= (q^2-1)v_{2i}, &
    x_2v_{0i} &= (q-1)v_{1i}, &
    x_2v_{2i} &= A' v_{0i}, \\
y_1v_{2i} &= (q^2-1)v_{1i}, &
    y_1v_{1i} &= y_1v_{0i}=0, \\
y_2v_{2i} &= (q-1)v_{1i}, &
    y_2v_{1i} &= y_2v_{0i}=0.
\end{align*}
As long as $V$ does not have any nontrivial subspace that is invariant both for $A$ and for $A'$, the resulting weight module $M$ is simple. For example we can take
\[A=
\begin{bmatrix}1 & 1 & 0 & \cdots & 0 \\
0 & 1 & 1 & \ddots & \vdots \\ 
\vdots & \ddots & \ddots &\ddots & 0 \\
\vdots & & \ddots & \ddots & 1 \\
0 & \cdots & \cdots & 0 & 1 \end{bmatrix} \qquad 
\quad\text{and}\quad
A'=\begin{bmatrix}
1 & 0 & \cdots & \cdots & 0 \\
1 & 1 & \ddots &  & \vdots \\ 
0 & 0 & \ddots &\ddots & \vdots \\
\vdots & & \ddots & \ddots & 0 \\
0 & \cdots & \cdots & 0 & 1 \end{bmatrix}.\]
\end{example}


\begin{thebibliography}{10}

\bibitem{B1}
V.~V.~Bavula.
\newblock Generalized {W}eyl algebras and their representations.
\newblock {\em Algebra i Analiz}, 4(1):75--97, 1992.

\bibitem{bavsimp}
V.~V.~Bavula.
\newblock Simple {$D[X,Y;\sigma,a]$}-modules.
\newblock {\em Ukrainian Mathematical Journal}, 44(12):1500--1511, 1992.

\bibitem{BR}
J.~Bell and D.~Rogalski.
\newblock {$\Bbb{Z}$}-graded simple rings.
\newblock {\em Trans. Amer. Math. Soc.}, 368(6):4461--4496, 2016.

\bibitem{DGO}
Y.~A.~Drozd, B.~L.~Guzner, and S.~A.~Ovsienko.
\newblock Weight modules over generalized {W}eyl algebras.
\newblock {\em J. Algebra}, 184(2):491--504, 1996.

\bibitem{GHo}
J.~Gaddis and P.~Ho.
\newblock Fixed rings of quantum generalized {W}eyl algebras.
\newblock {\em Comm. Algebra}, 48(9):4051--4064, 2020.

\bibitem{GR}
J.~Gaddis and D.~Rosso.
\newblock Fixed rings of twisted generalized {W}eyl algebras.
\newblock {\em J. Pure Appl. Algebra}, 227(4):Paper No. 107257, 30, 2023.

\bibitem{GW1}
J.~Gaddis and R.~Won.
\newblock Fixed rings of generalized {W}eyl algebras.
\newblock {\em J. Algebra}, 536:149--169, 2019.

\bibitem{Jprim}
D.~A.~Jordan.
\newblock Primitivity in skew {L}aurent polynomial rings and related rings.
\newblock {\em Math. Z.}, 213(3):353--371, 1993.

\bibitem{JW}
D.~A.~Jordan and I.~E.~Wells.
\newblock Invariants for automorphisms of certain iterated skew polynomial
  rings.
\newblock {\em Proc. Edinburgh Math. Soc. (2)}, 39(3):461--472, 1996.

\bibitem{KRS}
D.~S.~Keeler, D.~Rogalski, and J.~T.~Stafford.
\newblock Na\"\i ve noncommutative blowing up.
\newblock {\em Duke Math. J.}, 126(3):491--546, 2005.

\bibitem{KK}
E.~Kirkman and J.~Kuzmanovich.
\newblock Fixed subrings of {N}oetherian graded regular rings.
\newblock {\em J. Algebra}, 288(2):463--484, 2005.

\bibitem{kulk}
R.~S.~Kulkarni.
\newblock Down-up algebras and their representations.
\newblock {\em J. Algebra}, 245(2):431--462, 2001.

\bibitem{LMO}
A.~Leroy, J.~Matczuk, and J.~Okni\'{n}ski.
\newblock On the {G}el'fand-{K}irillov dimension of normal localizations and
  twisted polynomial rings.
\newblock In {\em Perspectives in ring theory ({A}ntwerp, 1987)}, volume 233 of
  {\em NATO Adv. Sci. Inst. Ser. C Math. Phys. Sci.}, pages 205--214. Kluwer
  Acad. Publ., Dordrecht, 1988.

\bibitem{Ztwist}
J.~J.~Zhang.
\newblock Twisted graded algebras and equivalences of graded categories.
\newblock {\em Proc. London Math. Soc. (3)}, 72(2):281--311, 1996.

\end{thebibliography}

\end{document}